%% file: secantMethod.tex
\documentclass[final,hidelinks,onefignum,onetabnum]{siamart220329}

\usepackage{amsfonts}
\usepackage{hyperref}
\usepackage{amsmath}
\usepackage{amssymb}
\usepackage{mathrsfs}
\usepackage{bm}
\usepackage[sans]{dsfont}
\usepackage{multirow}
\usepackage{subfigure}
\usepackage{enumerate}
\usepackage{xcolor}
\usepackage{url}

\ifpdf
\hypersetup{
  pdftitle={On convergence of the secant method},
  pdfauthor={Yan Tan, Chenhao Ye, Qinghai Zhang, Shubo Zhao},
  pdfproducer={pdflatex},
  colorlinks=true,         
  linkcolor=blue,          
  citecolor=green,         
  urlcolor=cyan            
}
\fi

\newsiamremark{notation}{Notation}
\newsiamremark{example}{Example}
\newsiamremark{hypothesis}{Hypothesis}

\headers{On convergence of the secant method}
{Yan Tan, Chenhao Ye, Qinghai Zhang, and Shubo Zhao}

\title{On Convergence of the Secant Method\thanks
  {
    Yan Tan and Chenhao Ye are co-first authors
    with equal contributions
    while Qinghai Zhang is the corresponding author. 
    \funding
    {
      This reseach was supported in part by
      a grant from the Ministry of Science and Technology of China.
    }
  }
}

\author{Yan Tan\footnotemark[1] \thanks
  {
    School of Mathematical Sciences,
    Zhejiang University, Hangzhou, Zhejiang, 310058, China.
    (\email{qinghai@zju.edu.cn})
  }
  \and 
  Chenhao Ye\footnotemark[1] \footnotemark[2]
  \and
  Qinghai Zhang\footnotemark[1] \footnotemark[2]
  \and 
  Shubo Zhao\footnotemark[1] 
  \thanks
  {
    College of Mathematics and System Sciences,
    Xinjiang University,
    Urumqi, Xinjiang, 830046, China.
  }
}

\usepackage{amsopn}
\DeclareMathOperator{\range}{range}
\DeclareMathOperator{\dom}{dom}

\newcommand{\st}{\ \text{  s.t. }\ }
\newcommand{\calB}{\mathcal{B}}
\newcommand{\calC}{\mathcal{C}}
\newcommand{\bbN}{\mathbb{N}}
\newcommand{\bbR}{\mathbb{R}}

\begin{document}

\maketitle

\begin{abstract}
  \input{sec/abstract}
\end{abstract}



\section{Introduction}
\label{sec:intro}
\input{sec/intro}

\section{Preliminaries and notation}
\label{sec:preliminary}
\input{sec/pre}

\section{Analysis of a simple root}
\label{sec:analysis}
\input{sec/analysis}

\section{Analysis of a multiple root}
\label{sec:analysisMR}
\input{sec/multipleRoots}

\section{Conclusion}
\label{sec:conclusion}
\input{sec/conclusion}



\bibliographystyle{siamplain}
\bibliography{bib/secant.bib}
\end{document}

%% file: sec/abstract.tex
The secant method, as an important approach for 
 solving nonlinear equations, 
 is introduced in nearly all numerical analysis textbooks. 
However, most textbooks only briefly address 
 the Q-order of convergence of this method, with few providing 
 rigorous mathematical proofs. 
This paper establishes a rigorous proof for 
 the Q-order of convergence of the secant method and 
 theoretically compares its computational 
 efficiency with that of Newton's method.

%% file: sec/intro.tex
Numerous well-established numerical methods exist for 
 finding roots of nonlinear equations $f(x)=0$. 
The widely used Newton's method employs the iterative scheme
\begin{equation*}
  \renewcommand{\arraystretch}{1.2}
  \begin{array}{ll}
    x_{n+1} = x_n - \frac{f(x_n)}{f'(x_n)}.
  \end{array}
\end{equation*}
This method can commence iterative computation directly 
 once an initial value $x_0$ is provided.
The secant method, a derivative-free variant of Newton's method, 
 replaces the derivative with a secant slope approximation 
 at the $n$th iteration
\[
  \renewcommand{\arraystretch}{1.2}
  \begin{array}{ll}
    f'(x_n) \leftarrow \frac{f(x_n) - f(x_{n-1})}{x_n - x_{n-1}},
  \end{array}
\]
 yielding the iterative scheme
\begin{equation}
    \label{eq:SecantFormula}
    \renewcommand{\arraystretch}{1.2}
    \begin{array}{ll}
        x_{n+1} = x_n - f(x_n)  \frac{x_n - x_{n-1}}{f(x_n) - f(x_{n-1})}.
    \end{array}
\end{equation}
Although this method requires two initial values, $x_0$ and $x_1$, 
 to initiate iterations, it remains widely applicable. 
This stems from two key advantages: 
 (i) enhanced robustness by eliminating derivative computations 
 compared to Newton's method, 
 and (ii) accelerated convergence relative to the bisection method.

Let $\alpha$ be a root of the equation $f(x)=0$. 
For the iterative sequence $\{x_n\}$, 
 we define the error sequence $\{e_n\}$ by $e_n = x_n - \alpha$. 
From a theoretical perspective, 
 we seek initial values $x_0, x_1$ such that 
 $\{x_n\}$ converges to $\alpha$. 
Furthermore, we wish to find the Q-order of convergence $p$ and 
 the corresponding asymptotic error constant $c$ for $\{x_n\}$, i.e., 
 to prove that $\lim\limits_{n\to\infty}\frac{E_{n+1}}{E_n^p}=c$, 
 where $E_n = |e_n|$, see Definition~\ref{def:p-orderConvergence}.

When $\alpha$ is a simple root, 
 although standard textbooks universally assert 
 Q-superlinear convergence of the secant method for simple roots,  
 rigorous proofs remain scarce. 
Two classical approaches exist: 
 asymptotic error analysis and the Fibonacci approach.  
D\'iez~\cite{diez2003note} and Kincaid et al.~\cite{kincaid2009numerical} 
 derived the quantitative relationship between 
 the Q-order of convergence of the secant method 
 and the golden ratio through asymptotic error analysis. 
However, this approach presupposes both 
 the convergence of the secant method 
 and sufficiently small errors at each step, 
 assumptions that may not hold without appropriate initial conditions.
Meanwhile, Atkinson~\cite{atkinson2008introduction} 
 and Gautschi~\cite{gautschi1997numerical} developed 
 the Fibonacci approach,
 which only estimates the R-order of convergence of the error,
 that is, the Q-order of convergence of this error upper bound,
 by tracking the error at each iteration.
While this approach clearly illustrates the error reduction process
 and the convergence of $\{x_n\}$,
 it does not directly address the Q-order of convergence of
 the actual error itself.
This distinction is crucial because the R-order of convergence of
 a sequence does not necessarily equal its Q-order of convergence,
 as demonstrated in the following example.

\begin{example}
  For $0 < k < 1$, define $y_n = k^{n + (-1)^n}$ and $z_n = k^{n-1}$. 
  Then for all  $n \geq 1$, we have $0 < y_n \leq z_n$, 
   so $\{z_n\}$ is an upper bound for $\{y_n\}$. 
  The sequence $\{z_n\}$ has Q-linear convergence since 
   $\frac{|z_{n+1}|}{|z_n|} = k$. 
  However, the ratio $\frac{|y_{n+1}|}{|y_n|}$ oscillates between $1$ 
   and $k^{2}$ and thus does not converge; 
   consequently, $\{y_n\}$ exhibits R-linear convergence
   but lacks Q-linear convergence.
\end{example}

When $\alpha$ is a root with multiplicity $m \ge 2$, 
 related research is even more limited. 
The classic approach remains the asymptotic error analysis 
 proposed by D\'iez~\cite{diez2003note}, 
 but it introduces further presuppositions. 
For $m = 2$, assuming the iterative sequence $\{x_n\}$ converges, 
 it proves Q-linear convergence 
 and finds the corresponding asymptotic error constant. 
For $m > 2$, it directly assumes that
 $\{x_n\}$ converges linearly to $\alpha$ 
 and then finds the corresponding asymptotic error constant. 
As mentioned in the analysis for simple roots, 
 these assumptions require appropriate initial conditions to hold. 
In fact, for some simple functions, 
 even with straightforward initial conditions, 
 the iteration may not converge, as shown in the following example.

\begin{example}\label{exm:xmNotConv}
  For the function $f(x)=x^{2}$, and $\alpha=0$
   is a $2$-fold root of $f$.
  The corresponding iterative formula for the secant method is 
   $x_{n+1}=x_{n} - \frac{x_n^2}{x_n + x_{n-1}}$.
  Let $k_n = \frac{e_{n+1}}{e_n} = \frac{x_{n+1}}{x_n}$, 
   the sequence $\{k_n\}$ satisfies the iterative formula 
   $k_{n}= 1 - \frac{k_{n-1}}{k_{n-1} + 1}$.
  If the initial ratio $k_0 = \tfrac{x_1}{x_0}$ equals the fixed point 
   $k^* := -\tfrac{1 + \sqrt{5}}{2}$, 
   then $k_n = k^*$ for all $n$, and thus $x_n = (k^*)^n x_0$. 
  Since $|k^*| > 1$, the sequence $\{x_n\}$ diverges, 
   despite the initial values $x_0, x_1$ being arbitrarily close 
   to the root $\alpha = 0$.  
\end{example}


Addressing the shortcomings in the current theory of the secant method, 
 we pose the following questions:
\begin{enumerate}
  \item[(Q-1)] 
    For simple roots, how can we rigorously prove the convergence of
     the secant iterative sequence $\{x_n\}$ near $\alpha$? 
    How can we rigorously compute the Q-order of convergence and 
     the corresponding asymptotic error constant for $\{x_n\}$?
  \item[(Q-2)] 
    For multiple roots, how can we provide 
     appropriate sufficient conditions such that 
     the secant iterative sequence $\{x_n\}$ near $\alpha$ converges? 
    Under such sufficient conditions, 
     how can we rigorously prove the Q-linear convergence of $\{x_n\}$ 
     and find the corresponding asymptotic error constant?
\end{enumerate}

In this paper, we provide positive answers to all above questions.

For (Q-1), Theorem~\ref{thm:secantConvergence} provides 
 a rigorous proof of Q-superlinear convergence.
It retains the advantage of explicit error tracking 
 while eliminating vague asymptotic analysis 
 and insufficient upper-bound estimates. 
By constructing the proof directly on the error terms 
 and leveraging properties of the Fibonacci sequence, 
 it rigorously establishes the convergence of $\{x_n\}$ 
 and explicitly derives the Q-order of convergence for the error itself.

For (Q-2), Section~\ref{sec:linearConvergenceMR} gives
  simple sufficient conditions 
 ($k_0 \in (0, 1) \cup (1, +\infty)$ 
 and $e_0 > 0$ is sufficiently small) to fully prove that 
 the iterative sequence $\{x_n\}$ converges linearly to $\alpha$
 and to determine the corresponding asymptotic error constant. 
Furthermore, in Section~\ref{sec:sufficientConditionsMR}, 
 we attempt to relax the sufficient conditions on the initial values 
 under which $\{x_n\}$ still converges linearly.
Theorem~\ref{thm:knMoreRangeOdd} and Theorem~\ref{thm:fx2mCommConvergence}
 provide sufficient (but not necessary) conditions 
 on the initial values $k_0$ and $e_0$ 
 for the Q-linear convergence of the iterative sequence $\{x_n\}$
 when $m$ is odd and even, respectively.

Thus, the main contributions of this work are:
\begin{itemize}
  \item 
    Providing rigorous proofs for the Q-order of convergence 
     and the existence of the corresponding asymptotic error constants 
     for $\{x_n\}$ in both simple and multiple root cases, 
     addressing the gaps in existing textbooks and literature.
  \item 
    For the multiple root case, 
     supplementing appropriate sufficient conditions on the initial values,
     serving as theoretical guidance for 
     selecting initial values in the secant method. 
    Practically, these conditions help 
     avoid choosing inappropriate initial values that 
     would lead to a divergent iterative sequence $\{x_n\}$.
\end{itemize}

The rest of this paper is structured as follows.
Section~\ref{sec:preliminary} gives some fundamental lemmas 
 and relevant notation;
Section~\ref{sec:analysis} presents the convergence analysis 
 for simple roots, including a comparative analysis of computation time 
 with Newton's method, fully addressing (Q-1);
Section~\ref{sec:analysisMR} provides the convergence analysis 
 for multiple roots, covering the Q-linear convergence of 
 the iterative sequence near multiple roots 
 and sufficient conditions on the initial values 
 to ensure convergence of the iterative sequence, 
 thereby addressing all aspects of (Q-2);
Section~\ref{sec:conclusion} concludes the paper.





%% file: sec/pre.tex
In this section,
we introduce some fundamental concepts to be used later.

\subsection{Q-order of convergence}
\label{sec:quot-conv-asympt}

A sequence $\{x_n\}$ is said to \emph{converge} to $L$ if
\begin{equation*}
  \forall \epsilon >0,\ \exists N\in \bbN, \ \text{ s.t. } \
  \forall n> N,\ |x_n - L| \leq \epsilon,
\end{equation*}
 denoted by $\lim\limits_{n \to \infty} x_n = L$.
The speed of convergence of $\{x_n\}$ to $L$
 is measured by

\begin{definition}[Q-order of convergence]
  \label{def:p-orderConvergence}
  A convergent sequence $\{x_n\}$ is said to \emph{converge} to $L$ 
   with \emph{Q-order} $p$ $(p \geq 1)$ if 
  \begin{equation*}
    \renewcommand{\arraystretch}{1.2}
    \begin{array}{ll}
      \lim\limits_{n \to \infty} 
      \frac{|x_{n+1} - L|}{|x_{n} - L|^p} = c > 0,
    \end{array}
  \end{equation*}
    where the constant $c$ is called the 
    \emph{asymptotic error constant (AEC)}.
   In particular, $\{x_n\}$ has \emph{Q-linear convergence}
    if $p = 1$ and $0 < c < 1$, 
    \emph{Q-superlinear convergence} if $1 < p < 2$,
    and \emph{Q-quadratic convergence} if $p = 2$.
\end{definition}

Throughout this paper, the term converges linearly is used 
 to mean Q-linear convergence, as no ambiguity arises in the present context.

\subsection{Fibonacci numbers}
\label{sec:fibonacci-numbers}

The following definition introduces Fibonacci numbers, 
 which are closely related to the Q-superlinear convergence of 
 the secant method.

\begin{definition}
  \label{def:Fibonacci}
  The sequence $\{F_n\}$ of \emph{Fibonacci numbers} is defined as
  \begin{equation*}
    F_0 = 0,\quad F_1 = 1, \qquad  F_{n+1} = F_n + F_{n-1}.
  \end{equation*}
\end{definition}

By analyzing the recurrence relation, 
 we derive its closed-form expression.

\begin{theorem}[Binet's formula]
  \label{thm:BinetFormula}
  Denote the golden ratio by $r_0 := \frac{1 + \sqrt{5}}{2}$
   and let $r_1 := 1 - r_0 = \frac{1 - \sqrt{5}}{2}$. 
  Then
  \begin{equation}
    \label{eq:BinetFormula}
    \renewcommand{\arraystretch}{1.2}
    \begin{array}{ll}
      F_n = \frac{r_0^n - r_1^n}{\sqrt{5}}.
    \end{array}
  \end{equation}
\end{theorem}


\begin{corollary}
  \label{cor:FibonacciIteration}
  The Fibonacci sequence in Definition~\ref{def:Fibonacci} satisfies
  \begin{equation*}
    F_{n+1} = r_0 F_n + r^n_1.
  \end{equation*}
\end{corollary}
\begin{proof}
  This follows from $r_1 = r_0 - \sqrt{5}$ and (\ref{eq:BinetFormula}).
\end{proof}

\subsection{Approximating rational functions}
\label{sec:appr-rati-funct}

We first introduce some necessary notation and preliminary results.

\begin{notation}[Range notation]
  \label{ntn:range}
  For a finite sequence $\{a_{i}\}_{i=1}^{n}$, define
  \begin{align*}
    \range(a_1, a_2, \ldots, a_n) &:= (\min(a_1, \ldots, a_n), \, 
                                       \max(a_1, \ldots, a_n)),\\
    \range[a_1, a_2, \ldots, a_n] &:= [\min(a_1, \ldots, a_n), \, 
                                       \max(a_1, \ldots, a_n)].
  \end{align*}
\end{notation}

\begin{notation}[Big-O notation]
  \label{ntn:highOrder}
  Let $f(x)$ and $g(x)$ be functions defined in a neighborhood 
   of a point $x_0$.
  We write $f(x) = O\left(g(x)\right)$ as $x \to x_0$ if there exist
   constants $C > 0$ and $\delta > 0$ such that for all $x$ with
   $|x - x_0| < \delta$, $|f(x)| \le C\,|g(x)|$.
  When no confusion arises, we omit "as $x \to x_0$" and simply write
   $f(x) = O\left(g(x)\right)$.
\end{notation}

\begin{theorem}[Taylor's theorem]
  \label{thm:taylor}
  Suppose that $f$ has continuous derivatives up to order $n+1$ 
   on an interval containing $a$. 
  Then, for any $x$ in this interval,
  \[
    \renewcommand{\arraystretch}{1.2}
    \begin{array}{ll} 
      f(x) = f(a) + f'(a)(x - a) + \frac{f''(a)}{2!}(x - a)^2 + \cdots 
             + \frac{f^{(n)}(a)}{n!}(x - a)^n + R_n(x),
    \end{array}
  \]
  where $f^{(k)}(a)$ denotes the $k$th derivative of $f$ at $a$, 
   $(x - a)^k$ is the $k$th power centered at $a$, 
   and for some $\xi \in \range(a, x)$,
  \[
    \renewcommand{\arraystretch}{1.2}
    \begin{array}{ll} 
      R_n(x) := \frac{f^{(n+1)}(\xi)}{(n+1)!}(x - a)^{n+1}.
    \end{array}
  \]
  Equivalently, $R_n(x) = O\left((x-a)^{n+1}\right)$ as $x \to a$.
\end{theorem}


\begin{theorem}[Lagrange's mean value theorem]
  \label{thm:meanValue}
  If a function $f \in \calC[a, b]$ and $f'$ exists on $(a, b)$, then
  \begin{displaymath}
    \renewcommand{\arraystretch}{1.2}
    \begin{array}{ll} 
      \exists \xi \in (a, b) \  \text{ s.t. } \quad
      f'(\xi) = \frac{f(b) - f(a)}{b - a}. 
    \end{array}
  \end{displaymath}
\end{theorem}

\begin{corollary}
  \label{coro:reverseRoll}
  Consider a function $f(x) \in \calC^1[a, b]$, $a \neq b$,
   and for any $x\in(a,b)$, $f'(x) \neq 0$.
  Then $f$ is an injective function,
   that is, for any $y, z \in [a, b] $, if $y \neq z $,
   then $f(y) \neq f(z)$.
\end{corollary}
\begin{proof}
  Assume for the sake of contradiction that $f(y) = f(z)$.
  By Theorem \ref{thm:meanValue},
   there exists
   $w \in \range(y, z) \subseteq (a, b)$ such that
  \[0 = f(y) - f(z) = f'(w)(y - z).\]
   Since $y \neq z$, we obtain $f'(w) = 0$,
   which contradicts $f'(x) \neq 0$ on $(a, b)$.
\end{proof}

\begin{theorem}[Cauchy's mean value theorem]
  \label{thm:cauchyMid}
  If functions $f, g \in \calC[a, b]$ and $f', g'$ exist on $(a, b)$, 
   and satisfy that for any $x \in (a, b) $, $g'(x) \neq 0$.
  Then there exists a point $ \xi \in (a, b) $ such that
  \[ 
    \renewcommand{\arraystretch}{1.2}
    \begin{array}{ll} 
      \frac{f(b) - f(a)}{g(b) - g(a)} = \frac{f'(\xi)}{g'(\xi)}. 
    \end{array}
  \]
\end{theorem}

\begin{lemma}
  \label{lem:epsLimit}
  Let functions $A(s): \mathcal{S} \to \bbR$ and
   $B(t): \mathcal{T} \to \bbR$ satisfy the following conditions:
  \begin{enumerate}
    \item $\exists L>0 \st \forall t \in \mathcal{T},\ |B(t)| \ge L$;
    \item $\exists M > 0 \st \forall s \in \mathcal{S}$,
          $\forall t \in \mathcal{T},\ \left|\frac{A(s)}{B(t)}\right| \le M$. 
  \end{enumerate}
  Then  
  \[ 
    \renewcommand{\arraystretch}{1.2}
    \begin{array}{ll} 
      \forall \epsilon > 0, \exists \delta > 0 \st  
      \forall |x|, |y| < \delta, \forall s \in \mathcal{S},
      \forall t \in \mathcal{T},
      \quad
      \left| \frac{A(s) + O(x)}{B(t) + O(y)} - \frac{A(s)}{B(t)} \right| 
    < \epsilon.
    \end{array}
  \]
\end{lemma}
\begin{proof}
  The definitions of $O(x) $ and $ O(y)$ yield
  \begin{align*}
      &\exists \delta_x > 0, M_x > 0 \st \forall |x| < \delta_x,\ 
       |O(x)| \le M_x |x|,\\
      &\exists \delta_y > 0, M_y > 0 \st \forall |y| < \delta_y,\ 
       |O(y)| \le M_y |y|.
  \end{align*}
  Combined with $|B(t)| \ge L > 0$ for all $t \in \mathcal{T}$, 
   we have
  \[
    \renewcommand{\arraystretch}{1.2}
    \begin{array}{ll} 
      \forall t\in\mathcal{T},\ 
      \forall |y|<\delta_0, \quad 
      |O(y)| \le M_y |y| < \frac{L}{2} \le \frac{|B(t)|}{2},
    \end{array}
  \] 
   where $\delta_0 := \min\left(\delta_y, \frac{L}{2M_y}\right)$.
  Hence the triangle inequality shows
  \[
    \renewcommand{\arraystretch}{1.2}
    \begin{array}{ll} 
      \forall t\in\mathcal{T},\ 
      \forall |y|<\delta_0,\quad 
      |B(t)+O(y)|\ge |B(t)|-|O(y)|>|B(t)|-\frac{|B(t)|}{2}=\frac{|B(t)|}{2}.
    \end{array}
  \]
  For any $\epsilon>0$, let 
   $\delta_1 := \frac{L\epsilon}{2(M + 1)(M_x + M_y)}$ and
   $\delta := \min\left(\delta_x, \delta_y, \delta_0, \delta_1\right)>0$. 
  Then for any $|x|, |y| < \delta$, $s \in \mathcal{S}$,
   and $t \in \mathcal{T}$, by the triangle inequality,
  \begin{displaymath}
    \renewcommand{\arraystretch}{1.2}
    \begin{array}{ll}
      &|B(t)O(x) - A(s)O(y)|
       \le |B(t)||O(x)|+|A(s)||O(y)|
      \\&\hspace{3.6cm}
       \le |B(t)|\cdot M_x|x|+|A(s)|\cdot M_y|y|
      \\&\hspace{3.6cm}
       < (|B(t)| M_x+|A(s)| M_y)\delta
      \\&\hspace{3.6cm}
       < (|A(s)|+|B(t)|)(M_x+M_y)\delta.
    \end{array}
  \end{displaymath}
  Therefore, it holds that
  \begin{displaymath}
    \renewcommand{\arraystretch}{2}
    \begin{array}{ll}
      &\left|\frac{A(s) + O(x)}{B(t) + O(y)} - \frac{A(s)}{B(t)}\right| 
       = \left|\frac{B(t)O(x) - A(s)O(y)}{B(t)(B(t)+O(y))}\right| 
       < \frac{(|A(s)|+|B(t)|)(M_x+M_y)\delta}{\frac{1}{2}|B(t)||B(t)|} 
      \\&\hspace{6.22cm}
       = \frac{2}{|B(t)|}\left(\left|\frac{A(s)}{B(t)}\right| + 1\right)
         (M_x + M_y)\delta,
    \end{array}
  \end{displaymath}
   which, together with 
   $\delta \leq \delta_1 = \frac{L\epsilon}{2(M + 1)(M_x + M_y)}$,
   implies
   $\left|\frac{A(s) + O(x)}{B(t) + O(y)} - \frac{A(s)}{B(t)}\right| 
    < \epsilon$.
\end{proof}

\begin{lemma}
  \label{lem:characterRoots0}
  For $m \geq 2$, the characteristic equation
  \begin{equation}
    \label{eq:characterEqm0}
    p_m(k) := k^{m} + k^{m-1} - 1 = 0
  \end{equation}
   has a unique solution $c_{m, 0}$ in $(0, 1)$.
\end{lemma}
\begin{proof}
  Since $-1 = p_m(0) < 0 < p_m(1) = 1$,
   the intermediate value theorem shows that
   \eqref{eq:characterEqm0} has a solution $c_{m,0}$ in $(0,1)$.
  For $k \in (0, 1)$ and $m \geq 2$, we have 
   $p_m'(k) = k^{m-2}(mk + m - 1) > 0$, 
   thus $p_m(k)$ is strictly monotonically increasing on $(0, 1)$,
   which ensures the uniqueness of $c_{m,0}$.
\end{proof}


%% file: sec/analysis.tex
This section establishes a rigorous proof for
 the Q-order of convergence of the secant method
 in the case of a simple root.

Let $\alpha$ denote the root of function $f(x)$, 
 $\{x_n\}$ be the sequence of iterates near $\alpha$,
 and 
 \begin{align}
   e_n &:= x_n - \alpha,\label{eq:iterateSignError}\\
   E_n &:= |x_n - \alpha|\label{eq:iterateNosignError}
 \end{align}
 be the iteration errors. 
\begin{hypothesis} 
  \label{hyp:nonTerminate}
  We exclude the trivial case in which the iteration reaches the root 
   in finitely many steps; that is, 
   we assume $x_n \neq \alpha$ for all $n \in \bbN$. Consequently,
  \begin{equation}
    \label{eq:xnNotEqualalpha}
    \forall n \in \bbN, \quad x_n \neq \alpha, 
    \quad e_n \neq 0, \quad \text{and} \quad E_n > 0.
  \end{equation}
\end{hypothesis}

Given this hypothesis, the following lemma establishes that
 the iteration of the secant method \eqref{eq:SecantFormula}
 is well-defined under certain conditions.

\begin{lemma}
  \label{lem:simpleEnNoteq}
  Consider a function $f(x) \in \calC^{2}(\calB)$
   with domain $\calB := [\alpha - \delta, \alpha + \delta]$,
   and $\alpha$ is a simple root of $f$,
   i.e., $f$ satisfies $f(\alpha) = 0$
   and $f'(\alpha) \neq 0$.
  Then there exists $\delta_0 \in (0, \delta)$ such that
   for any $x_{n-1} \neq x_n$ satisfying
   $x_{n-1}, x_n \in \calB_0 := [\alpha - \delta_0, \alpha + \delta_0]$, 
   we have $f(x_{n-1}) \neq f(x_{n})$,
   and the secant method iteration~(\ref{eq:SecantFormula}) yields
   $x_{n+1} \neq x_{n}$.
\end{lemma}
\begin{proof}
  The continuity of $f'$ and the assumption $f'(\alpha) \neq 0$ imply
  \begin{equation*}
     \exists \delta_0 \in (0, \delta) \ 
     \text{ s.t. } \
     \forall x \in \calB_0, \ f'(x) \neq 0.
  \end{equation*}
  Together with Corollary~\ref{coro:reverseRoll}, $x_{n-1} \neq x_n$, 
   and $x_{n-1}, x_n \in \calB_0$, we have $f(x_{n-1}) \neq f(x_{n})$. 
  Similarly, the assumption \eqref{eq:xnNotEqualalpha} 
   and $x_{n}, \alpha \in \calB_0$
   yield that $f(x_{n}) \neq f(\alpha) = 0$. 
  According to the secant method iteration
   \eqref{eq:SecantFormula} and $x_{n-1} \neq x_n$, we obtain
  \[
    \renewcommand{\arraystretch}{1.2}
    \begin{array}{ll}
      |x_{n+1} - x_n| = \left|\frac{f(x_n)}{f(x_{n-1}) - f(x_{n})}
                              (x_{n-1} - x_n)\right|
                      \ge \frac{|f(x_n)|}{|f(x_{n-1})| + |f(x_{n})|}
                        |x_{n-1} - x_n| 
                      > 0.
    \end{array}
  \] 
\end{proof}

Next, we analyze the error relation of the secant method
 for simple roots.

\begin{lemma}[Error relation of the secant method for simple roots]
  \label{lem:secantErrorRelation}
  Consider a function $f(x) \in \calC^{2}(\calB)$,
   and $\alpha$ is a simple root of $f$.
  If $x_{n-1} \neq x_{n}$,
   then for the secant method~(\ref{eq:SecantFormula}), 
   there exist $\xi_n \in \range(x_{n-1}, x_n)$
   and $\zeta_n \in \range(x_{n-1}, x_n, \alpha)$ such that
  \begin{equation*}
    \renewcommand{\arraystretch}{1.2}
    \begin{array}{ll} 
        x_{n+1} - \alpha 
      = (x_n - \alpha)(x_{n-1} - \alpha)\frac{f''(\zeta_n)}{2f'(\xi_n)}.
    \end{array}
  \end{equation*}
\end{lemma}

\begin{proof}
  Applying algebraic manipulations 
   to the secant method (\ref{eq:SecantFormula}) yields
  \begin{equation}
    \label{eq:secantErrorRelation2}
    \renewcommand{\arraystretch}{1.2}
    \begin{array}{ll} 
        x_{n+1} - \alpha
      = (x_n - \alpha)(x_{n-1} - \alpha)
        \frac{\frac{f[x_{n-1}, x_n] - f[x_n, \alpha]}{x_{n-1} - \alpha}}
             {f[x_{n-1}, x_n]},
    \end{array}
   \end{equation}
   where $f[a, b] := \frac{f(a) - f(b)}{a - b}$.
  By Theorem~\ref{thm:meanValue} and the condition $x_{n-1} \neq x_{n}$,
   there exists $\xi_n \in \range(x_{n-1}, x_n)$ such that 
  \begin{equation}
    \label{eq:DiffToDeri1}
    \renewcommand{\arraystretch}{1.2}
    \begin{array}{ll} 
        f[x_{n-1}, x_n] 
      = \frac{f(x_{n-1}) - f(x_n)}{x_{n-1} - x_n} 
      = f'(\xi_n).
    \end{array}
  \end{equation}
  Define a function $g(x) := f[x, x_n]$. 
  Using assumption \eqref{eq:xnNotEqualalpha} and
   applying Theorem~\ref{thm:meanValue} to $g(x)$, it follows that 
  \[
    \renewcommand{\arraystretch}{1.2}
    \begin{array}{ll} 
      \frac{f[x_{n-1}, x_n] - f[\alpha, x_n]}{x_{n-1} - \alpha} 
      = \frac{g(x_{n-1}) - g(\alpha)}{x_{n-1} - \alpha}
      = g'(\beta) 
    \end{array}
  \]
   for some $\beta \in \range(x_{n-1}, \alpha)$.
  Computing the derivative of $g(\beta)$ and using 
   Theorem~\ref{thm:taylor}, we have
  \begin{equation}
    \label{eq:DiffToDeri2}
    \renewcommand{\arraystretch}{1.2}
    \begin{array}{ll} 
        g'(\beta) 
      = (f[\beta, x_n])' 
      = \frac{f(x_n) - f(\beta) - (x_n - \beta)f'(\beta)}{(x_n - \beta)^2}
      = \frac{f''(\zeta_n)}{2}
    \end{array}
  \end{equation}
   for some $\zeta_n \in \range(\beta, x_{n}) 
                     \subseteq \range(x_{n-1}, x_n, \alpha)$.
  The proof is completed by substituting (\ref{eq:DiffToDeri1}) 
   and (\ref{eq:DiffToDeri2}) into (\ref{eq:secantErrorRelation2}). 
\end{proof}

The error relation in Lemma~\ref{lem:secantErrorRelation} 
 establishes the Q-superlinear convergence of the secant method
 in the case of a simple root.

\begin{theorem}[Convergence of the secant method for simple roots]
  \label{thm:secantConvergence}
  Consider a function $f(x) \in \calC^{2}(\calB)$,
   and $\alpha$ is a simple root of $f$.
  If both $x_0$ and $x_1$ are chosen sufficiently close to 
   the root $\alpha$, $x_0 \neq x_1$ and $f''(\alpha) \neq 0$, 
   then the sequence of iterates $\{x_n\}$ in the secant method
   converges to $\alpha$ 
   with Q-order $p = r_0 \approx 1.618$ 
   and AEC $c = m_\alpha^{r_0 - 1}$, 
   i.e.,
  \begin{equation*}
    \renewcommand{\arraystretch}{1.2}
    \begin{array}{ll} 
        \lim\limits_{n \to \infty} 
        \frac{E_{n+1}}{E_{n}^{r_0}} 
      = m_\alpha^{r_0 - 1},
    \end{array}
  \end{equation*}
   where $E_{n}$ is as defined in \eqref{eq:iterateNosignError}
   and $m_\alpha := \left|\frac{f''(\alpha)}{2f'(\alpha)} \right| > 0$.
\end{theorem}

\begin{proof}
  We first prove that the sequence $\{x_n\}$ generated 
   by the secant method converges to the root $\alpha$.
   
  Define 
   $M := \frac{\max_{x \in \calB_0} |f''(x)|}
   {2 \min_{x \in \calB_0} |f'(x)|}$,
   where $\calB_0$ is as described in Lemma \ref{lem:simpleEnNoteq}.
  $M$ is well-defined because $\min_{x \in \calB_0} |f'(x)| > 0$
   and $\max_{x \in \calB_0} |f''(x)| < +\infty$,
   and we have $M>0$ from $f''(\alpha) \neq 0$.
  Since both $x_0$ and $x_1$ are chosen sufficiently close to $\alpha$,
   we can let $\delta_1 := \min(\delta_0, \frac{1}{M}) \in (0, \delta)$ and
  \[E_0 = |x_0 - \alpha| < \delta_1, \quad E_1 = |x_1 - \alpha| < \delta_1. \]
  The assumption \eqref{eq:xnNotEqualalpha} gives that
   for all $\ell \in \bbN$, $E_\ell > 0$.
  Moreover, we claim that
   for all $\ell \in \bbN$, $E_\ell < \delta_1$,
   and for all $\ell \in \bbN^+$, $x_{\ell-1} \neq x_\ell$.
  We already have that $E_0, E_1 < \delta_1$ and $x_0 \neq x_1$.
  Assume by the induction hypothesis that
   $E_{n-1}, E_{n} < \delta_1$ and $x_{n-1} \neq x_n$.
  Then it suffices to show
   $E_{n+1} < \delta_1$ and $x_n \neq x_{n+1}$.
  On the one hand,
   it follows that
   \[E_{n+1} = E_n E_{n-1} \frac{|f''(\zeta_n)|}{2|f'(\xi_n)|}
            \leq E_n E_{n-1} M < E_n 
            < \delta_1, \]
   where the first step follows from Lemma~\ref{lem:secantErrorRelation},
   the second step from the definition of $M$
   and $\zeta_n, \xi_n \in \range(x_{n-1}, x_n, \alpha)
        \subseteq (\alpha - \delta_1, \alpha + \delta_1) \subseteq \calB_0$,
   and the third step from $M E_{n-1} < M \delta_1 \le 1$.
  On the other hand, $E_{n-1}, E_{n} < \delta_1$ 
   implies $x_{n-1}, x_n \in \calB_0$,
   which, together with Lemma \ref{lem:simpleEnNoteq}, 
   yields $x_n \neq x_{n+1}$.
  Therefore, the induction step is completed,
   and we get
  \begin{equation}
    \label{eq:secantErrorRelation}
    E_{n+1} = E_n E_{n-1} m_n,
  \end{equation}
   where $m_n := \left|\frac{f''(\zeta_n)}{2f'(\xi_n)} \right|$. 
  Multiplying both sides of (\ref{eq:secantErrorRelation}) by $M$ gives
  \begin{equation}
    \label{eq:secantErrorMEInequality}
    M E_{n+1} = M E_n E_{n-1} m_n \leq M E_n M E_{n-1}.
  \end{equation}
  Define $\eta := \max(M E_0, M E_1)$; we have $\eta < M \delta_1 \le 1$.
  Then induction on (\ref{eq:secantErrorMEInequality}) shows that
  \begin{equation*}
     \begin{aligned}
       & ME_{n} \le ME_{n-1} ME_{n-2} < \eta^{F_n + F_{n-1}} 
        = \eta^{F_{n+1}},
     \end{aligned}
  \end{equation*}
   where $\{F_n\}$ is the Fibonacci sequence as in 
   Definition~\ref{def:Fibonacci}.
  Thus, \mbox{$E_n < \frac{1}{M} \eta^{F_{n+1}}$}.
  From Theorem~\ref{thm:BinetFormula} and $|r_1| < 1$, 
   it holds that $F_n \sim \frac{r_0^{n}}{\sqrt{5}}$ as 
   $n \to \infty$. 
  Since $\eta < 1$, we have $\lim\limits_{n\to \infty} E_n = 0$,
   which
   shows that $\{x_n\}$ converges to $\alpha$.

  Next, we prove that the convergence order is $p = r_0$
  and the AEC is $c = m_\alpha^{r_0 - 1}$. 

  By performing induction on the recurrence relation
   \eqref{eq:secantErrorRelation}, 
   we obtain the closed-form expression for $E_n$ as follows:
   \begin{align*}
     &E_n = E_1^{F_n} E_0^{F_{n-1}} m_1^{F_{n-1}} \cdots 
     m_{n-2}^{F_2} m_{n-1}^{F_1}, \\
     &E_{n+1} = E_1^{F_{n+1}} E_0^{F_n} m_1^{F_n} \cdots m_{n-1}^{F_2} 
     m_n^{F_1}.
   \end{align*}
  Hence
  \begin{align*}
    \renewcommand{\arraystretch}{1.2}
    \begin{array}{ll} 
     &\frac{E_{n+1}}{E_n^{r_0}} 
     = E_1^{F_{n+1}-r_0F_n} E_0^{F_n-r_0F_{n-1}} 
        m_1^{F_n-r_0F_{n-1}} \cdots 
        m_{n-1}^{F_2 - r_0 F_1} m_n^{F_1} 
    \\ & \hspace{.88cm}    
     = E_1^{r_1^n} E_0^{r_1^{n-1}} 
        m_1^{r_1^{n-1}} \cdots m_{n-1}^{r_1^1} m_n^1,
    \end{array}
  \end{align*}
   where the second step follows from 
   Corollary~\ref{cor:FibonacciIteration}.  
  Since $\{x_n\}$ converges to $\alpha$, this yields 
   $\lim\limits_{n \to \infty} m_n = m_\alpha > 0$, 
   which means
  \begin{equation}
    \label{eq:m_nRange}
    \renewcommand{\arraystretch}{1.2}
    \begin{array}{ll} 
      \exists N\in \bbN
      \ \text{ s.t. } \ 
      \forall n>N, \ m_n \in 
      \left(\frac{1}{l} m_\alpha, l m_\alpha\right),
    \end{array} 
  \end{equation}
   where $l > 1$ is an arbitrary constant. 
  We define
  \begin{align*}
    & A_n := E_1^{r_1^n} \cdot E_0^{r_1^{n-1}} \cdot 
        m_1^{r_1^{n-1}} \cdot m_2^{r_1^{n-2}} 
        \cdots m_{N}^{r_1^{n-N}}, \\ 
    & B_n := m_{N+1}^{r_1^{n-N-1}} \cdot m_{N+2}^{r_1^{n-N-2}} \cdots  
        m_{n-1}^{r_1^1} \cdot m_n^1
  \end{align*}
   so that $\frac{E_{n+1}}{E_n^{r_0}} = A_n B_n$.
  Then $|r_1| < 1$ implies $\lim\limits_{n \to \infty} A_n = 1$.
  As for $B_n$, (\ref{eq:m_nRange}) gives
  \begin{align*}
    \renewcommand{\arraystretch}{1.2}
    \begin{array}{ll} 
      &B_n 
      \leq 
      \begin{cases}
        (lm_\alpha)^1\left(\frac{1}{l}m_\alpha \right)^{r_1} 
          \cdots \left(lm_\alpha\right)^{r_1^{n-N-1}} 
        \hspace{1cm} \text{if $n - N -1$ is even}; \\
      (lm_\alpha)^1\left(\frac{1}{l}m_\alpha \right)^{r_1} 
        \cdots \left(\frac{1}{l}m_\alpha\right)^{r_1^{n-N-1}} 
        \hspace{0.85cm} \text{if $n - N -1$ is odd} 
     \end{cases}  
     \\ & \hspace{0.55cm}
     = l^{1 - r_1 + r_1^2 - \cdots + (-r_1)^{n-N-1}} 
        \cdot m_\alpha^{1+r_1 + r_1^2 + \cdots + r_1^{n-N-1}}
     \\ & \hspace{0.55cm}
     = l^{\sum_{i=0}^{n-N-1}(-r_1)^i} 
        \cdot m_\alpha^{\sum_{i=0}^{n-N-1}r_1^i},
    \end{array} 
  \end{align*}
   where the discrimination of the even-numbered factors in 
   the first line comes from the fact that $r_1 < 0$. 
  Taking $n\to\infty$ in the above inequality yields
  \begin{equation}
     \label{eq:BnLimitle}
     \lim\limits_{n\to\infty}B_n 
     \leq l^{\frac{1}{1+r_1}} m_\alpha^{\frac{1}{1-r_1}}
     = l^{\frac{1}{1+r_1}} m_\alpha^{r_0 - 1}. 
  \end{equation}
  By similar arguments, we derive
  \begin{equation}
     \label{eq:BnLimitge}
     \lim\limits_{n\to\infty}B_n 
     \geq l^{-\frac{1}{1+r_1}} m_\alpha^{r_0 - 1}. 
  \end{equation}
  Since inequalities \eqref{eq:BnLimitle} and \eqref{eq:BnLimitge} hold
   for any $l>1$, taking $l\to 1^+$ in both implies
   $\lim\limits_{n\to\infty}B_n=m_\alpha^{r_0 - 1}$.
  Therefore,
  \begin{align*}
    \renewcommand{\arraystretch}{1.2}
    \begin{array}{ll} 
      \lim\limits_{n\to\infty} \frac{E_{n+1}}{E_n^{r_0}} 
      = \lim\limits_{n\to\infty} A_{n} \lim\limits_{n\to\infty} B_{n} 
      = m_\alpha^{r_0-1},
    \end{array}
  \end{align*}
   which, together with Definition~\ref{def:p-orderConvergence},
   shows that the secant method converges with Q-order $p = r_0$ and 
   the AEC $c = m_\alpha^{r_0-1}$.
\end{proof}

Furthermore, the application of asymptotic analysis 
 to the secant method has also demonstrated 
 analogous conclusions \cite{kincaid2009numerical}.

Here, to compare the secant method with Newton's method, 
 we present the Q-order of convergence of Newton's method without proof. 
In fact, most numerical analysis textbooks already provide detailed and
 rigorous mathematical proofs for the Q-order of convergence 
 of Newton's method.

\begin{theorem}[Convergence of Newton's method for simple roots] 
  \label{thm:NewtonConvergence}
  Consider a function $f(x)\in \calC^{2}(\calB)$,
   and $\alpha$ is a simple root of $f$.
  If $x_0$ is chosen sufficiently close to the root $\alpha$,
   then the sequence of iterates $\{x_n\}$ in Newton's method 
   converges to $\alpha$ with Q-order $p = 2$ 
   and AEC $c = m_\alpha$,
    i.e.,
  \begin{equation*}
    \renewcommand{\arraystretch}{1.2}
    \begin{array}{ll} 
      \lim\limits_{n \to \infty}
      \frac{E_{n+1}}{E_n^2}
      = m_\alpha.
    \end{array}
  \end{equation*}
\end{theorem}

Clearly, in terms of the Q-order of convergence, 
 Newton's method outperforms the secant method. 
However, this does not imply that 
 for the same problem and the same accuracy requirement, 
 Newton's method requires less computational time than the secant method. 
Hence, we compare the time required by
 Newton's method and the secant method
 to achieve the same accuracy requirement.

\begin{corollary}
  \label{cor:NewtonSecantEfficiency}
  Consider solving $f(x) = 0$ near a root $\alpha$. 
  Let $m$ and $s \cdot m$ be the time to evaluate $f(x)$ 
   and $f'(x)$, respectively.
  If their initial values are chosen sufficiently close to the root $\alpha$,
   then the minimum time required to obtain the desired absolute 
   accuracy $\epsilon$ with Newton's method 
   and the secant method are respectively given by
  \begin{align*}
    T_N &= (1 + s)m \lceil \log_2 K \rceil,  \\
    T_S &= m \lceil \log_{r_0} K \rceil, 
  \end{align*}
   where 
      $K := \frac{\log (m_\alpha \epsilon)}
                 {\log (m_\alpha E_0)}$
   and $\lceil \cdot \rceil$ denotes the rounding-up operator,
   which maps a real number $x$ to 
   the smallest integer greater than or equal to $x$.
\end{corollary}

\begin{proof}
  See \cite[Sec.~2.3]{atkinson2008introduction}.
\end{proof}

Therefore, for the same problem and precision requirements, 
 if the secant method outperforms Newton's method in computational time,
 i.e., $\frac{T_S}{T_N} < 1$, we can derive 
 the condition $s > \frac{\log 2}{ \log r_0} -1 \approx 0.44$.


%% file: sec/multipleRoots.tex
We adopt the notation and Hypothesis~\ref{hyp:nonTerminate}
 introduced at the beginning of Section~\ref{sec:analysis}.
Without loss of generality, let $\alpha = 0$; otherwise, consider the
 function $\hat{f}(x) := f(x + \alpha)$ instead of $f(x)$. 
Therefore, by \eqref{eq:iterateSignError} and \eqref{eq:SecantFormula},
 the error $e_n = x_n - \alpha = x_n$ satisfies the secant method iteration
\begin{equation}
  \label{eq:secantMethod}
  \renewcommand{\arraystretch}{1.2}
  \begin{array}{ll}
    e_{n+1} = e_n - f(e_n)\frac{e_{n-1} - e_{n}}{f(e_{n-1}) - f(e_{n})}.
  \end{array}
\end{equation} 
Let $\calB = [-\delta, \delta]$ be
 a neighborhood containing the zero of the function $f(x)$.

\subsection{The Q-linear convergence near a multiple root}
\label{sec:linearConvergenceMR}
\input{sec/linearConvergenceMR}

\subsection{Sufficient conditions on initial values
  for Q-linear convergence}
\label{sec:sufficientConditionsMR}  
\input{sec/sufficientConditionsMR}


%% file: sec/linearConvergenceMR.tex
To properly characterize the Q-linear convergence of $\{e_n\}$,
 i.e., the limit $\lim\limits_{n \to \infty} \frac{E_{n+1}}{E_{n}}
 = \lim\limits_{n \to \infty} \frac{|e_{n+1}|}{|e_n|}$ exists
 and belongs to $(0, 1)$, we define 
\begin{equation}
  \label{eq:knForm}
  \renewcommand{\arraystretch}{1.2}
  \begin{array}{ll}
    k_n := \frac{e_{n+1}}{e_{n}}.
  \end{array}
\end{equation}
The assumption \eqref{eq:xnNotEqualalpha} ensures that
 $k_n$ is well-defined and $k_n \neq 0$.
Since $e_0 = x_0$ and $k_0 = \frac{e_1}{e_0} = \frac{x_1}{x_0}$,
 there is a bijection between the initial values $x_0, x_1$
 and $k_0, e_0$.
Thus, we may equivalently take $k_0, e_0$ as the initial values
 and study the conditions for the convergence of the secant method
 error sequence $\{e_n\}$ in the neighborhood $\calB$.

We first assume that $e_0$ and $e_1$ have the same sign,
 so that $k_0 > 0$.
For the case $k_0 < 0$, 
 Theorems~\ref{thm:knMoreRangeOdd}~and~\ref{thm:fx2mCommConvergence} 
 show that if $k_0$ and $e_0$ satisfy appropriate conditions,
 analogous convergence results still hold.
Without loss of generality, we further assume $e_0 > 0$; otherwise,
 one may consider the function $\tilde{f}(x) := f(-x)$,
 which reduces the problem to the case of $e_0 > 0$.

We will prove that if the initial values $k_0, e_0$ satisfy certain
 conditions, then the sequence $\{k_n\}$ converges to $c_{m, 0} \in (0, 1)$,
 where $c_{m, 0}$ is as described in Lemma~\ref{lem:characterRoots0}. 

First, the following lemma shows that under certain conditions,
 the iteration formula \eqref{eq:secantMethod} is well-defined.
\begin{lemma}
  \label{lem:enNoteq}
  Consider a function $f(x) \in \calC^{m+1}(\calB)$
   with $m\ge 2$,
   and $0$ is an $m$-fold root of $f$,
   i.e., $f$ satisfies $j=0, 1, \ldots, m-1,\ f^{(j)}(0)=0$
   and $f^{(m)}(0) \neq 0$.
  Then there exists $\delta_0 \in (0, \delta]$ such that
   for any $ x \in (0, \delta_0)$, $f'(x) \neq 0$.
  Furthermore, for $e_{n-1} \neq e_n$ satisfying
   $e_{n-1}, e_n \in (0, \delta_0)$, we have $f(e_{n-1}) \neq f(e_n)$, 
   and the secant method iteration \eqref{eq:secantMethod}
   yields $e_{n+1} \neq e_n$.
\end{lemma}
\begin{proof}
  Theorem~\ref{thm:taylor}
   and $j=0, 1, \ldots, m-1,\ f^{(j)}(0)=0$ imply that
   for any $x\in\calB$, there exists
   $\xi_x \in \range(0, x) \subseteq \calB$ such that
  \begin{equation*}
    \renewcommand{\arraystretch}{1.2}
    \begin{array}{ll}
      f'(x) = x^{m-1}\left(\frac{f^{(m)}(0)}{(m-1)!}
              + \frac{f^{(m+1)}(\xi_x)}{m!} x  \right).
    \end{array}
  \end{equation*}
  Since $f(x) \in \calC^{m+1}(\calB)$, it follows that
   for $i=0, 1, \ldots, m+1$, $f^{(i)}(x)$ is continuous on $\calB$.
  Combined with the fact that $\calB$ is a closed interval,
   the absolute value of $f^{(i)}(x)$ has a finite nonzero upper bound
  \begin{equation}
      \label{eq:MfmPrimeForm}
      \forall i=0, 1, \ldots, m+1, \quad
      M_{f, \calB}^{(i)} := \max_{x\in \calB}
       \left| f^{(i)}(x)\right| + 1.
  \end{equation}
  Thus
  \[
    \renewcommand{\arraystretch}{1.2}
    \begin{array}{ll}
      \forall x\in\calB,\
      \left| f^{(m+1)}(\xi_x)\right| < M_{f, \calB}^{(m+1)}
       \Longrightarrow 
       \forall x\in\calB,\
      1 - \frac{|f^{(m+1)}(\xi_x)|}{ M_{f, \calB}^{(m+1)}} > 0.
    \end{array}
  \]
  Since $ f^{(m)}(0) \neq 0$, 
   set $\delta_0 = 
   \min\left(\frac{m\left|f^{(m)}(0)\right|}{
   M_{f, \calB}^{(m+1)}}, \delta\right) \in (0,\delta]$.
  Then for any $x \in (0, \delta_0)$, by the triangle inequality, 
  \begin{displaymath}
  \renewcommand{\arraystretch}{1.5}
    \begin{array}{ll}
      &\left|\frac{f^{(m)}(0)}{(m-1)!}
             + \frac{f^{(m+1)}(\xi_x)}{m!} x  \right| 
       \ge \left|\frac{f^{(m)}(0)}{(m-1)!}  \right|
           - \left|\frac{f^{(m+1)}(\xi_x)}{m!} x  \right| 
      \ge \left|\frac{f^{(m)}(0)}{(m-1)!}  \right|
           - \left|\frac{f^{(m+1)}(\xi_x)}{m!} \right| \delta_0
    \\ & \hspace{3.46cm}
      \ge \frac{|f^{(m)}(0)|}{(m-1)!}
           - \frac{|f^{(m+1)}(\xi_x)|\cdot|f^{(m)}(0)|}
                  {M_{f, \calB}^{(m+1)}(m-1)!}  
    \\ & \hspace{3.46cm}
      = \frac{|f^{(m)}(0)|}{(m-1)!} 
          \cdot \left( 1 - \frac{|f^{(m+1)}(\xi_x)|}{
           M_{f, \calB}^{(m+1)}}\right)
      > 0,
    \end{array}
  \end{displaymath}
   which gives $|f'(x)| = |x|^{m-1}\left|\frac{f^{(m)}(0)}{(m-1)!}
   + \frac{f^{(m+1)}(\xi_x)}{m!} x \right| > 0$. 

  For $e_{n-1} \neq e_n$ satisfying
   $e_{n-1}, e_n \in (0, \delta_0) \subseteq [0, \delta_0]$,
   by Corollary~\ref{coro:reverseRoll}, combined with the condition that
   for any $x \in (0, \delta_0),\ f'(x) \neq 0$,
   we have $f(e_{n-1}) \neq f(e_n)$.  
  Hence, $0 < |f(e_{n-1}) - f(e_n)| < 2 M_{f,\calB}^{(0)}$,  
   where $M_{f,\calB}^{(0)}$ is defined as in \eqref{eq:MfmPrimeForm}.
  Similarly, for $e_n \neq 0$ satisfying
   $e_n, 0 \in [0, \delta_0]$, we have $f(e_n) \neq f(0) = 0$.  
  According to the secant method recurrence formula
   \eqref{eq:secantMethod} and $e_{n-1} \neq e_n$, we obtain
  \[
    \renewcommand{\arraystretch}{1.2}
    \begin{array}{ll}
      |e_{n+1} - e_n| = \left|\frac{f(e_n)}{f(e_{n-1}) - f(e_{n})}
                              (e_{n-1} - e_n)\right|
                      > \frac{|f(e_n)|}{2  M_{f,\calB}^{(0)}} |e_{n-1} - e_n| 
                      > 0.
    \end{array}
  \]
\end{proof}

Next, we prove the convergence of $\{e_n\}$.
\begin{theorem}
  \label{thm:mEnConvergence}
  Consider a function $f(x) \in \calC^{m+1}(\calB)$ with $m\ge 2$,
   and $0$ is an $m$-fold root of $f$.
  If the initial values $k_0$ and $e_0$ are chosen such that 
   $k_0 \in (0, 1)\cup (1, +\infty)$ and
   $e_0 > 0$ is sufficiently small,
   then the secant method error $e_n > 0$, and $\{e_n\}$ converges to $0$.
\end{theorem}
\begin{proof} 
  Algebraic manipulation of the secant method recurrence formula
   \eqref{eq:secantMethod} yields 
  \begin{displaymath}
    \renewcommand{\arraystretch}{1.7}
    \begin{array}{ll}
      &e_{n+1} 
      = e_n - f(e_n)\frac{e_{n-1} - e_{n}}{f(e_{n-1}) - f(e_{n})} 
      = e_{n}\frac{f(e_{n-1}) - f(e_n) - \frac{f(e_n)(e_{n-1} - e_{n})}{e_n}}
      {f(e_{n-1}) - f(e_n)} 
      \\ & \hspace{0.8cm}
      = e_n \frac{f(e_{n-1}) - \frac{e_{n-1}}{e_n} f(e_n)}
      {f(e_{n-1}) - f(e_n)}.
    \end{array}
  \end{displaymath}
  Hence \eqref{eq:knForm} gives
  \begin{equation}
    \label{eq:knRecm}
    \renewcommand{\arraystretch}{1.2}
    \begin{array}{ll}
      k_n = \frac{f(e_{n-1}) - \frac{f(e_{n-1}k_{n-1})}{k_{n-1}}}
                 {f(e_{n-1}) - f(e_{n-1}k_{n-1})}
          = e_{n-1} \frac{\frac{f(e_{n-1})}{e_{n-1}} - \frac{f(e_{n})}{e_{n}}}
    {f(e_{n-1}) - f(e_{n})}.
    \end{array}
  \end{equation}
  By Theorem~\ref{thm:taylor} and the fact that
   for $j = 0, 1, \ldots, m-1,\ f^{(j)}(0) = 0$, we have 
  \begin{align}
    \label{eq:ftaylor}
    f(x) &= x^{m}\left(\tfrac{f^{(m)}(0)}{m!} + O(x)\right),\\
    \label{eq:fprimetaylor}
    f'(x) &= x^{m-1}\left(\tfrac{f^{(m)}(0)}{(m-1)!} + O(x)\right),\\
    \label{eq:fsecondprimetaylor}
    f''(x) &= x^{m-2}\left(\tfrac{f^{(m)}(0)}{(m-2)!} + O(x)\right),
  \end{align}
   where each $O(x)$ term satisfies the existence of a constant
   $ M_{f, \calB}^{(m+1)} > 0$ as defined in \eqref{eq:MfmPrimeForm}
   such that $|O(x)| \leq M_{f, \calB}^{(m+1)} |x|$.

  Since $k_0 > 0$, we have $e_1 = k_0 e_0 > 0$.  
  Note that for $\ell \in \{0, 1\}$, we have
   $0 < e_\ell \le \eta := \max(e_0, e_1) = \max(e_0, k_0 e_0) = O(e_0)$.
  Since $|e_0|$ is assumed to be sufficiently small,
   it follows that $\eta$ is also sufficiently small.
  Hence, we let $\eta \in (0, \delta_0)$,
   where $\delta_0$ is defined as in Lemma~\ref{lem:enNoteq}.  
  Moreover, $k_0 \neq 1$ implies $e_0 \neq e_1$.
  We claim that 
  \begin{equation}
    \label{eq:enRange}
    \forall \ell \in \bbN,\ e_\ell \in (0, \eta], \text{ and } 
    \forall \ell \in \bbN^+,\ e_{\ell-1} \neq e_{\ell}.
  \end{equation}

  The claim holds for $\ell = 1$ and $e_0 \in (0, \eta]$.
  Assume it holds for indices up to $n \geq 1$; that is,
   $e_{n-1}, e_n \in (0, \eta] \subseteq (0, \delta_0)$ and $e_{n-1} \neq e_n$.
  Then Lemma~\ref{lem:enNoteq} ensures that for all
   $x \in (0, \delta_0),\ f'(x) \neq 0$ and $e_{n} \neq e_{n+1}$.
  Thus, it suffices to show $e_{n+1} \in (0, \eta]$.
  By Theorem~\ref{thm:cauchyMid} and \eqref{eq:knRecm},
   there exists $t_n \in \range(e_{n}, e_{n-1})$ such that
  \begin{equation}
    \label{eq:knAndptn}
    \renewcommand{\arraystretch}{1.2}
    \begin{array}{ll}
        k_n 
      = \frac{e_{n+1}}{e_n} 
      = e_{n-1} \frac{\left(\frac{f}{x}\right)'(t_n)}{f'(t_n)}
      = e_{n-1} \frac{t_nf'(t_n) - f(t_n)}{t_n^2 f'(t_n)} = e_{n-1} p(t_n),
    \end{array}
  \end{equation}
   where $p(x) := \frac{x f'(x) - f(x)}{x^2f'(x)}$.
  Substituting \eqref{eq:ftaylor} and \eqref{eq:fprimetaylor} into $p(x)$
   yields
  \begin{equation*}
    \renewcommand{\arraystretch}{1.2}
    \begin{array}{ll}
        p(x) 
      = \frac{m x^m f^{(m)}(0) - x^m f^{(m)}(0) + O(x^{m+1})}
             {m x^{m+1}f^{(m)}(0) + O(x^{m+2})} 
      = \frac{m-1}{mx} \cdot \frac{1 + O(x)}{1 + O(x)}.
    \end{array}
  \end{equation*}
  Lemma~\ref{lem:epsLimit} shows that for
  \[
    \renewcommand{\arraystretch}{1.2}
    \begin{array}{ll}
      \epsilon = \frac{1}{2m} > 0,\ 
      \exists \delta_1 > 0 \st \forall |x| < \delta_1,
      \quad \left|\frac{1 + O(x)}{1 + O(x)} - 1\right|<\epsilon.
    \end{array}
  \]
  Combined with $t_n \in \range(e_{n-1}, e_n) \subseteq (0, \eta)$ 
   and sufficiently
   small $\eta > 0$ such that $|t_n| < \eta < \delta_1$, we obtain
   $\left|\frac{1 + O(t_n)}{1 + O(t_n)}-1\right|<\epsilon$. 
  Therefore
  \begin{equation}
    \label{eq:ptnRange}
    \renewcommand{\arraystretch}{1.2}
    \begin{array}{ll}
      p(t_n) \in \left(\frac{m-1}{mt_n}(1 - \epsilon),  
                       \frac{m-1}{mt_n}(1 + \epsilon)\right). 
    \end{array}
  \end{equation}
  Taking the derivative of $p(x)$ gives
  \begin{displaymath}
    \renewcommand{\arraystretch}{1.3}
    \begin{array}{ll}
      &p'(x) 
      = \frac{x^3f'(x)f''(x) -2x^2f'(x)f'(x) + 2xf(x)f'(x) 
                   - x^3f'(x)f''(x) + x^2f(x)f''(x)}{(x^2 f'(x))^2}
      \\ & \hspace{.85cm}
      = \frac{-2x(f'(x))^2+2f(x)f'(x)+xf(x)f''(x)}{x^3(f'(x))^2}.
    \end{array}
  \end{displaymath}
  Substituting \eqref{eq:ftaylor}, \eqref{eq:fprimetaylor}, and
   \eqref{eq:fsecondprimetaylor} into $p'(x)$ yields
  \begin{displaymath}
    \renewcommand{\arraystretch}{1.3}
    \begin{array}{ll}
      &p'(x) 
      = \frac{(-2m^2 x^{2m-1} + 2m x^{2m-1}+m(m-1)x^{2m-1})(f^{(m)}(0))^2 
        + O(x^{2m})}
             {m^2 x^{2m+1}(f^{(m)}(0))^2 + O(x^{2m+2})} 
      \\ & \hspace{.85cm}
      = \frac{1}{x^2}\cdot\frac{1-m+O(x)}{m+O(x)}.
    \end{array}
  \end{displaymath}
  Lemma~\ref{lem:epsLimit} shows that for
  \[
    \renewcommand{\arraystretch}{1.2}
    \begin{array}{ll}
      \epsilon=\frac{1}{2m}>0,\ \exists \delta_2>0 \st \forall |x|<\delta_2,
      \quad \left|\frac{1-m+O(x)}{m+O(x)}-\frac{1-m}{m}\right|<\epsilon.
    \end{array}
  \]
  For $x \in \range(e_{n-1}, e_n) \subseteq (0, \eta)$,
   since $\eta > 0$ is sufficiently small,
   it follows that $|x|< \eta <\delta_2$ and
  \[
    \renewcommand{\arraystretch}{1.2}
    \begin{array}{ll}
      p'(x) < \frac{1}{x^2}\left(\frac{1-m}{m}+\epsilon\right)
            = \frac{1}{x^2}\cdot\frac{3-2m}{2m}<0.
    \end{array}
  \]
  Hence $p(x)$ is strictly monotonically decreasing on
   $\range(e_{n-1}, e_{n})$,
   which implies $p(t_n) \in \range (p(e_{n-1}), p(e_{n}))$. 
  Furthermore, together with \eqref{eq:ptnRange}, we have
  \begin{align*}
    \renewcommand{\arraystretch}{1.2}
    \begin{array}{ll}
      p(t_n) \in 
      \left(
        \frac{m-1}{m}(1 - \epsilon)
        \min\left(\frac{1}{e_{n-1}}, \frac{1}{e_n} \right),
        \frac{m-1}{m}(1 + \epsilon)
        \max\left(\frac{1}{e_{n-1}}, \frac{1}{e_n} \right)
      \right).
    \end{array}
  \end{align*}
  Therefore, \eqref{eq:knAndptn} gives 
  \begin{subequations}
    \begin{gather}
    \label{eq:knRange}
    \tfrac{e_{n+1}}{e_n} 
    \in
    \left(\tfrac{m-1}{m}(1 - \epsilon)
          \min\left( 1, \tfrac{e_{n-1}}{e_n}\right),
          \tfrac{m-1}{m}(1 + \epsilon)
          \max\left( 1, \tfrac{e_{n-1}}{e_n} 
    \right)
    \right),\\
    \label{eq:kknRange}
    \tfrac{e_{n+1}}{e_{n-1}} 
    \in 
    \left(\tfrac{m-1}{m}(1 - \epsilon)
          \min\left( \tfrac{e_n}{e_{n-1}}, 1 \right),
          \tfrac{m-1}{m}(1 + \epsilon)
          \max\left( \tfrac{e_n}{e_{n-1}}, 1 \right)
    \right).
    \end{gather}
  \end{subequations}
  Let $\mu:=\frac{m-1}{m}(1+\epsilon)$.
  When $e_n > e_{n-1}$, \eqref{eq:knRange} shows that 
  \[
    \renewcommand{\arraystretch}{1.2}
    \begin{array}{ll}
      \frac{e_{n+1}}{e_n} \in 
      \left( \frac{m-1}{m} \frac{e_{n-1}}{e_n} (1 - \epsilon),
      \frac{m-1}{m} (1 + \epsilon) \right),
    \end{array}
  \]
   which implies $ 0 < e_{n+1} < \mu e_n \le \mu \max(e_{n-1}, e_n)$.
  Similarly, when $e_n < e_{n-1}$, by \eqref{eq:kknRange}
   we get $ 0 < e_{n + 1} < \mu e_{n - 1} \le \mu \max(e_{n - 1}, e_n)$.
  Hence, it always holds that
   $0 < e_{n+1} < \mu \max(e_{n-1}, e_{n}) \le \mu \eta$.
  Note that for $m \ge 2$, we have
  \[ 
    \renewcommand{\arraystretch}{1.2}
    \begin{array}{ll}
      \mu = \frac{m-1}{m}(1+\epsilon) = \frac{2m^2-m-1}{2m^2} \in (0,1),
    \end{array}
  \]
   which yields $e_{n+1} \in (0, \eta]$.
  Thus, the induction step is completed, and we obtain
   $0 < e_{n+1} < \mu\max(e_{n-1}, e_{n}) $.
  It follows that
  \begin{align*}
    e_{2n}   &< \mu\max(e_{2n-2}, e_{2n-1}), \\
    e_{2n+1} &< \mu\max(e_{2n-1}, e_{2n}) 
              \leq \mu \max(e_{2n-2}, e_{2n-1}),
  \end{align*}
   which means $\max(e_{2n},e_{2n+1})\le \mu\max(e_{2n-2},e_{2n-1})$.
  By mathematical induction, we get
   $\max(e_{2n}, e_{2n+1}) \le \mu^n \max(e_0, e_1) = \mu^n \eta$. 
  Therefore
  \[
        \lim\limits_{n\to+\infty} e_{2n} 
    \le \lim\limits_{n\to+\infty} \mu^n \eta = 0, \quad
        \lim\limits_{n\to+\infty} e_{2n+1} 
    \le \lim\limits_{n\to+\infty} \mu^n \eta = 0.
  \]
  Combined with $e_n>0$, we conclude that $\{e_n\}$ converges to $0$.
\end{proof}

When the subscript $\ell$ is sufficiently large,
 the following lemma shows that 
 $k_\ell$ belongs to some closed subinterval of $(0, 1)$.

\begin{lemma}
  \label{lem:knRange}
  Consider a function $f(x) \in \calC^{m+1}(\calB)$ with $m \ge 2$,
   and $0$ is an $m$-fold root of $f$.
  If the initial values $k_0$ and $e_0$ are chosen such that 
   $k_0 \in (0, 1)\cup (1, +\infty)$ and
   $e_0 > 0$ is sufficiently small,
   then there exist $N > 0$ and $r \in \left(\frac{2m-3}{2m}, 1\right)$ 
   such that for any $\ell > N$, it holds 
   that $ k_\ell \in  \left[\frac{2m-3}{2m}, r\right]$, 
   where $k_{\ell}$ is as defined in \eqref{eq:knForm}.
\end{lemma}
\begin{proof}
  By Theorem~\ref{thm:mEnConvergence}, 
    we know that $e_n > 0$ and $e_n \to 0$,
    hence $k_n > 0$ for all $n \in \bbN$. 
  Furthermore, there must exist some
   $N \in \bbN$ such that $k_N < 1$.
  Indeed, if $k_n \ge 1$ for all $n \in \bbN$, 
   then $e_{n+1} \ge e_n$,
   so $\{e_n\}$ would be non-decreasing and bounded below by $e_0 > 0$,
   contradicting $e_n \to 0$.
  Substituting $n = N + 1$ into \eqref{eq:knRecm}, we obtain 
  \[
    \renewcommand{\arraystretch}{1.2}
    \begin{array}{ll}
        k_{N+1} 
      = \frac{f(e_{N}) - \frac{f(e_{N}k_{N})}{k_{N}}}{f(e_{N}) - f(e_{N}k_{N})}
      = 1 + \frac{f(e_{N}k_{N}) 
        - \frac{f(e_{N}k_{N})}{k_{N}}}{f(e_{N}) - f(e_{N}k_{N})}. 
    \end{array}
  \]
  Therefore, there exist
   $\hat{k}_N \in (k_N, 1)$, $\xi_N \in (0, e_N k_N)$, 
   and $\zeta_N \in (0, e_N \hat{k}_N)$
   such that
  \begin{displaymath}
    \renewcommand{\arraystretch}{1.7}
    \begin{array}{ll}
      &1 - k_{N+1} 
      = \frac{\frac{f(e_{N} k_{N})}{k_{N}} - f(e_{N} k_{N})}
              {f(e_{N}) - f(e_{N} k_{N})}
      = \frac{\frac{1}{k_N}(1 - k_N) f(e_N k_N)}
              {(1 - k_N) e_N f'(e_N \hat{k}_N)} 
      \\ & \hspace{1.52cm}
      = \frac{1}{k_Ne_N} \cdot 
         \frac{\frac{f^{(m)}(0)}{m!}(k_Ne_N)^m 
               + \frac{f^{(m+1)}(\xi_N)}{(m+1)!}(k_Ne_N)^{m+1} }
              {\frac{f^{(m)}(0)}{(m-1)!}(\hat{k}_N e_N)^{m-1} 
               + \frac{f^{(m+1)}(\zeta_N)}{m!}(\hat{k}_Ne_N)^{m} }  
      \\ & \hspace{1.52cm}
      = \left( \frac{k_N}{\hat{k}_N} \right)^{m-1} 
          \frac{1 + k_N \frac{f^{(m+1)}(\xi_N)}{(m+1)f^{(m)}(0)} e_N}
               {m + \hat{k}_N \frac{f^{(m+1)}(\zeta_N)}{f^{(m)}(0)} e_N}
      = \left( \frac{k_N}{\hat{k}_N} \right)^{m-1} 
        \frac{1 + O(e_N)}{m  + O(e_N)},
    \end{array}
  \end{displaymath}
   where the second step follows from
   $f(e_N) - f(e_N k_N) = (e_N - e_N k_N) f'(e_N \hat{k}_N)$;
   the third step holds because $k_N < 1$ allows
   canceling the nonzero common factor $1 - k_N$
   from the numerator and denominator,
   followed by applying the Taylor expansion of $f(e_N k_N)$
   and $f'(e_N\hat{k}_N)$ at $0$ using Theorem \ref{thm:taylor};
   and the last step follows from 
   $\left| k_N \frac{f^{(m+1)}(\xi_N)}{(m+1)f^{(m)}(0)} \right| \le 
    \frac{M_{f, \calB}^{(m+1)}}{(m+1)|f^{(m)}(0)|}$
   and
   $\left| \hat{k}_N \frac{f^{(m+1)}(\zeta_N)}{f^{(m)}(0)} \right| \le 
    \frac{M_{f, \calB}^{(m+1)}}{|f^{(m)}(0)|}$.
  Lemma~\ref{lem:epsLimit} shows that for
  \[
    \renewcommand{\arraystretch}{1.2}
    \begin{array}{ll}
      \epsilon = \frac{1}{2m} > 0,\ \exists \delta_1 > 0 
      \st \forall |e| < \delta_1,
      \quad \left| \frac{1 + O(e)}{m + O(e)} - \frac{1}{m} \right| < \epsilon.
    \end{array}
  \]
  From \eqref{eq:enRange} and the fact 
   that $\eta$ is sufficiently small so that $|e_N| \leq \eta < \delta_1$,
   we obtain $\left|\frac{ 1 + O(e_N)}{m  + O(e_N)} - \frac{1}{m}\right| 
   < \epsilon$.
  This implies
  \begin{equation}
    \label{eq:1MinuskNBound}
    \renewcommand{\arraystretch}{1.7}
    1 - k_{N+1}\left\{
    \begin{array}{l}
      > \left( \frac{k_N}{\hat{k}_N} \right)^{m-1} 
        \left( \frac{1}{m} - \epsilon \right)
      =  \left( \frac{k_N}{\hat{k}_N}\right)^{m-1} \frac{1}{2m} 
      > 0; \\
      < \left( \frac{k_N}{\hat{k}_N} \right)^{m-1} 
        \left( \frac{1}{m} + \epsilon \right) 
      = \left( \frac{k_N}{\hat{k}_N}\right)^{m-1}\frac{3}{2m} 
      < \frac{3}{2m},
    \end{array}
    \right.
  \end{equation}
   which is equivalent to $\frac{2m-3}{2m} < k_{N+1} < 1$. 

  Let $r := \max\left( k_{N+1}, 1- \frac{(2m-3)^{m-1}}{(2m)^m} \right) 
       \in \left(\frac{2m-3}{2m}, 1\right)$. 
  We claim that for all $\ell> N$,
   $k_\ell \in \left[ \frac{2m-3}{2m}, r \right]$. 
  For $\ell=N+1$, $k_{N+1} \leq r$ gives $\frac{2m-3}{2m} \leq k_{N+1} \leq r$.
  Assume by the induction hypothesis that for $\ell = n > N$,
   $k_n \in \left[ \frac{2m-3}{2m}, r\right]$.
  Then for $\ell = n + 1$, similar to the derivation 
   of \eqref{eq:1MinuskNBound}, we have
  \[
    \begin{array}{ll}
        \left( \frac{k_n}{\hat{k}_n} \right)^{m-1} \frac{1}{2m} 
      < 1 - k_{n+1} 
      < \left( \frac{k_n}{\hat{k}_n} \right)^{m-1} \frac{3}{2m},
    \end{array}
  \]
   where $\hat{k}_n \in (k_n, 1)$.
  Thus, on the one hand,
  \[ 
    \begin{array}{ll}
      k_{n+1} > 1 - \left( \frac{k_n}{\hat{k}_n} \right)^{m-1} \frac{3}{2m}
              > 1 - \frac{3}{2m} = \frac{2m-3}{2m}.
    \end{array}
  \]
   On the other hand, $\frac{2m-3}{2m} \leq k_n < \hat{k}_n < 1$ yields
   $\frac{k_n}{\hat{k}_n}> \frac{2m-3}{2m} $ and
  \[
    \begin{array}{ll}
      k_{n+1} < 1 -  \left( \frac{k_n}{\hat{k}_n} \right)^{m-1} \frac{1}{2m}
              < 1 - \left( \frac{2m-3}{2m} \right)^{m-1} \frac{1}{2m}
              = 1- \frac{(2m-3)^{m-1}}{(2m)^m} \leq r. 
    \end{array}
  \]
  Therefore, $k_{n+1}\in \left[ \frac{2m-3}{2m}, r\right]$,
   and the induction holds.
\end{proof}

Next, we prove that $\{k_n\}$ converges. 
Motivated by fixed-point theory, 
 we analyze the recurrence~\eqref{eq:knRecm}. 
For any $\ell, n \in \bbN$, we write 
\begin{equation*}
  \renewcommand{\arraystretch}{1.7}
  \begin{array}{ll}
    &k_{\ell +1} - k_{n+1} 
    = \frac{f(e_{\ell }) - \frac{f(e_{\ell }k_{\ell })}{k_{\ell }}}
            {f(e_{\ell }) - f(e_{\ell }k_{\ell })} 
      - \frac{f(e_{n}) - \frac{f(e_{n}k_{n})}{k_{n}}}
            {f(e_{n}) - f(e_{n}k_{n})}  
    \\ & \hspace{1.94cm}
    = \left( \frac{f(e_{\ell }) - \frac{f(e_{\ell }k_{\ell })}{k_{\ell }}}
                   {f(e_{\ell }) - f(e_{\ell }k_{\ell })} 
            - \frac{f(e_{\ell }) - \frac{f(e_{\ell }k_{n})}{k_{n}}}
                    {f(e_{\ell }) - f(e_{\ell }k_{n})}  \right)
    +  \left( \frac{f(e_{\ell }) - \frac{f(e_{\ell }k_{n})}{k_{n}}}
                   {f(e_{\ell }) - f(e_{\ell }k_{n})} 
            - \frac{f(e_{n}) - \frac{f(e_{n}k_{n})}{k_{n}}}
                   {f(e_{n}) - f(e_{n}k_{n})}  \right)   
    \\ & \hspace{1.94cm}
    = (d_\ell (k_\ell ) - d_\ell (k_n)) + (g_n(e_\ell ) - g_n(e_n)),
  \end{array}
\end{equation*}
 where 
\begin{equation}
  \label{eq:dlgnForm}
  \renewcommand{\arraystretch}{1.2}
  \begin{array}{l}
    d_\ell (k) := \frac{f(e_{\ell }) - \frac{f(e_{\ell }k)}{k}}
                    {f(e_{\ell }) - f(e_{\ell }k)}, 
    \quad 
    g_n(e) := \frac{f(e) - \frac{f(ek_{n})}{k_{n}}}{f(e) - f(ek_{n})}.
  \end{array}
\end{equation}
By Theorem \ref{thm:meanValue}, there exist 
  $\hat{k}_{\ell ,n} \in \range(k_{\ell }, k_{n}), 
   \hat{e}_{\ell ,n} \in \range(e_{\ell }, e_{n})$ such that
\begin{equation}
  \label{eq:fixKn}
  k_{\ell +1} - k_{n+1} = d_\ell ' (\hat{k}_{\ell ,n}) (k_\ell  - k_{n}) 
                        + g_{n}'(\hat{e}_{\ell ,n}) (e_{\ell } - e_{n}).
\end{equation}
Furthermore, the definition of $d_\ell$ implies
\begin{equation}
  \label{eq:dlklrelation}
  \forall \ell \in \bbN,\quad k_{\ell + 1}=d_\ell (k_\ell).
\end{equation}

The following lemma shows that, for sufficiently small $e_\ell$,
 the derivative $d_\ell '(k)$ is uniformly bounded away from $1$ 
 in absolute value.
\begin{lemma}
  \label{lem:dnPrimeBound}
  Consider a function $f(x) \in \calC^{m+1}(\calB)$ with $m \ge 2$, 
   and $0$ is an $m$-fold root of $f$.
  For any fixed $r \in \left(\frac{2m-3}{2m}, 1\right)$,
   there exists $\gamma_0 \in(0, \delta)$ such that
   for all $ k \in \left[ \frac{2m-3}{2m}, r \right]$,
   $e_\ell  \in (0, \gamma_0)$, 
   the derivative $d_\ell '(k)$ satisfies
   $ |d_\ell '(k)| < \frac{m}{m + 1}$,
   where $d_{\ell}(k)$ is defined in \eqref{eq:dlgnForm}.
\end{lemma}
\begin{proof}
  Taking the derivative of $d_\ell (k)$ yields
  \begin{displaymath}
    \renewcommand{\arraystretch}{1.7}
    \begin{array}{ll}
      &d_\ell '(k) 
      = \frac{\frac{1}{k^2}(e_\ell k^2 f(e_\ell)f'(e_\ell k) 
              - e_\ell k f(e_\ell) f'(e_\ell k) + f(e_\ell) f(e_\ell k) 
              - f(e_\ell k) f(e_\ell k))}{(f(e_\ell) - f(e_\ell k))^2} 
      \\ & \hspace{0.9cm}
      = \frac{ (m  k^{m-1} - m k^{m-2} 
        + k^{m-2} - k^{2m-2})e_\ell^{2m} 
        + O(e_\ell^{2m+1})}{(1 - k^{m})^2e_\ell^{2m} 
        + O(e_\ell^{2m+1})} 
      \\ & \hspace{0.9cm}
      = \frac{m k^{m-1} - (m - 1) k^{m-2} - k^{2m -2} + O(e_\ell)}{(1 - k^m)^2 
        + O(e_\ell)} 
    \end{array}
  \end{displaymath}
   where the second step applies Taylor expansions 
   to $f(e_\ell)$, $f(e_\ell k)$, and $f'(e_\ell k)$ 
   at $0$ using equations \eqref{eq:ftaylor} and \eqref{eq:fprimetaylor},
   and then absorbs the $O((e_\ell k)^{j})$ terms into $O(e_\ell^{j})$ 
   since $k < 1$.

  When $m = 2$, we have $k \ge \frac{1}{4}$. By Lemma~\ref{lem:epsLimit} 
   for $\epsilon = \frac{2}{75}$, there exists $\gamma_0 \in (0, \delta)$ 
   such that whenever $e_\ell \in(0, \gamma_0)$, we have
   $|d_\ell '(k)| = \frac{1 + O(e_\ell)}{(1 + k)^2 +O(e_\ell)} 
    < \frac{4^2}{5^2} + \epsilon = \frac{2}{3}=\frac{m}{m+1}$, 
   which directly proves the result. 

  For $m \geq 3$, define
  \begin{equation}
    \label{eq:bkForm}
    \begin{array}{l}
      b(k) := \frac{m k^{m-1} - (m - 1) k^{m-2} - k^{2m -2}}{(1 - k^m)^2},
    \end{array}
  \end{equation}
  We next prove that for $k \in \left[\frac{2m-3}{2m}, r\right]$, 
   $|b(k)| < \frac{m-1}{m}$. 
  Further simplifying $b(k)$ yields 
  \begin{displaymath}
    \renewcommand{\arraystretch}{1.2}
    \begin{array}{ll}
      & b(k) 
      = \frac{m k^{m-1} - (m - 1) k^{m-2} - k^{2m -2}}{(1 - k^m)^2} 
      = -\frac{k^{m-2}(1-k)(m - 1 - k\sum_{j=0}^{m-2} k^j) }{(1-k^m)^2} 
      \nonumber 
      \\ & \hspace{0.73cm}
      = -\frac{k^{m-2}(1-k)\sum_{j=1}^{m-1}(1 - k^j)}{(1-k^m)^2} 
      = -\frac{k^{m-2}(1-k)^2(\sum_{i=0}^{m-2} \sum_{j=0}^{i} k^j)}
              {(1-k)^2(\sum_{j=0}^{m-1} k^j)^2}
      \\ & \hspace{0.73cm}
      = -\frac{k^{m-2}(\sum_{i=0}^{m-2} \sum_{j=0}^{i} k^j)}
              {(\sum_{j=0}^{m-1} k^j)^2},
    \end{array}
  \end{displaymath}
   where the last step follows from the condition $k \leq r < 1$,
   by canceling the common factor $(1 - k)^2$
   in both the numerator and denominator.
  Observe that for $k > 0$, 
   we have $k^{m-2}(\sum_{i=0}^{m-2} \sum_{j=0}^{i} k^j) > 0$ 
   and $(\sum_{j=0}^{m-1} k^j)^2 > 0$, so $b(k) < 0$.
  Thus, it suffices to show that $b(k) > -\frac{m-1}{m}$. 
  By bounding $b(k)$, we obtain
  \begin{align*}
    \renewcommand{\arraystretch}{1.2}
    \begin{array}{ll}
        b(k)
      = -\frac{k^{m-2}\sum_{i=0}^{m-2} (\sum_{j=0}^{i} k^j)}
              {(\sum_{j=0}^{m-1} k^j)(\sum_{j=0}^{m-1} k^j)}
      > -\frac{k^{m-2}\sum_{i=0}^{m-2} (\sum_{j=0}^{m-1} k^j)}
              {(\sum_{j=0}^{m-1} k^j)(\sum_{j=0}^{m-1} k^j)}
      = -\frac{(m-1)k^{m-2}}{\sum_{j=0}^{m-1} k^j}.
    \end{array}
  \end{align*}
  Given that $0 < k < 1$, the inequality $ k^{m-3} + k^{m-1} > 2k^{m-2}$ holds,
   and we derive
  \begin{equation}
    \label{eq:bkRelease}
    \renewcommand{\arraystretch}{1.2}
    \begin{array}{ll}
        b(k)
      > -\frac{(m-1)k^{m-2}}{\sum_{j=0}^{m-1} k^j}
      > -\frac{(m-1)k^{m-2}}{\sum_{j=0}^{m-4} k^j + 3k^{m-2}} 
      > -\frac{(m-1)k^{m-2}}{mk^{m-2}} = -\frac{m-1}{m},
    \end{array}
  \end{equation}
   where the third step follows from that for
   $j = 0, 1, \ldots, m-4$, we have $k^j > k^{m-2}$.
 
  Since for $k \in \left[ \frac{2m-3}{2m}, r \right]$,
   $|b(k)| < \frac{m-1}{m}$ 
   and $\left| (1 - k^m)^2 \right| \ge \left| (1 - r^m)^2 \right| > 0$, 
  Lemma~\ref{lem:epsLimit} shows that for $\epsilon = \frac{1}{m(m+1)} > 0$,
   there exists $\gamma_0 \in (0, \delta)$ such that 
   whenever $e_\ell \in (0, \gamma_0)$, $|d_\ell '(k) - b(k)| < \epsilon$.
  Together with the triangle inequality, we have
  \[ 
    \renewcommand{\arraystretch}{1.2}
    \begin{array}{ll}
      |d_\ell '(k)| \le |b(k)| + |d_\ell '(k) - b(k)|
                    < \frac{m-1}{m} + \epsilon = \frac{m}{m+1}. 
    \end{array}
  \]
\end{proof}

Similarly, the following lemma shows that 
 for $k_n \in \left[ \frac{2m-3}{2m}, r \right]$, 
 and a sufficiently small $e$, the derivative $|g_n'(e)|$ is uniformly bounded.
\begin{lemma}
  \label{lem:gnPrimeBound}
  Consider a function $f(x) \in \calC^{m+1}(\calB)$ with $m\ge 2$,
   and $0$ is an $m$-fold root of $f$.
  For any fixed $r \in \left(\frac{2m-3}{2m}, 1\right)$, 
   define the constant
  \begin{equation}
    \label{eq:GForm}
    \renewcommand{\arraystretch}{1.2}
    \begin{array}{ll}
      G := \frac{4 M_{f, \calB}^{(m+1)}}
                {(1-r^m)\left| f^{(m)}(0)\right|} + 1 
        > 0,
    \end{array}
  \end{equation}
   where $M_{f, \calB}^{(m+1)}$ is as in \eqref{eq:MfmPrimeForm}.
  There exists $\gamma_1 \in (0, \delta)$ such that 
   for all $k_n \in \left[ \frac{2m-3}{2m}, r \right]$, $e \in (0, \gamma_1)$, 
   the derivative $g_n'(e)$ satisfies $|g_n'(e)| < G$,
   where $g_n(e)$ is defined in \eqref{eq:dlgnForm}.
\end{lemma}
\begin{proof}
  Taking the derivative of $g_n(e)$ yields
  \begin{align*}
    \renewcommand{\arraystretch}{1.2}
    \begin{array}{ll}
      g_n'(e) &= \frac{\left( \frac{1}{k_n} - 1 \right) 
                       \left( f'(e) f(e k_n) - k_n f(e) f'(e k_n) \right)}
                      {\left( f(e) - f(e k_n) \right)^2}.
    \end{array}
  \end{align*}
  Consider the Taylor expansions of $f(e), f'(e), f(e k_n)$, 
   and $f'(e k_n)$ at $0$. 
  By Theorem \ref{thm:taylor}, there exist 
   $\xi_0, \xi_1 \in (0, e)$ and 
   $\zeta_0, \zeta_1 \in (0, e k_n) \subseteq (0, e)$ 
   such that:
  \begin{displaymath}
    \renewcommand{\arraystretch}{1.2}
    \begin{array}{rl}
      &f(e) = \frac{f^{(m)}(0)}{m!} e^{m} 
            + \frac{f^{(m+1)}(\xi_0)}{(m+1)!}  e^{m+1},
    \\ & \hspace{-0.08cm}
      f'(e) = \frac{f^{(m)}(0)}{(m-1)!} e^{m-1} 
             + \frac{f^{(m+1)}(\xi_1)}{m!} e^{m},
    \\ & \hspace{-0.34cm}
      f(e k_n) = \frac{f^{(m)}(0)}{m!} e^{m} k_n^m 
                + \frac{f^{(m+1)}(\zeta_0)}{(m+1)!} e^{m+1} k_n^{m+1},
    \\ & \hspace{-0.42cm}
      f'(e k_n) = \frac{f^{(m)}(0)}{(m-1)!} e^{m-1} k_n^{m-1} 
                 + \frac{f^{(m+1)}(\zeta_1)}{m!} e^{m} k_n^m.
    \end{array}
  \end{displaymath} 
  From the above four equations, we compute
  \begin{displaymath}
    \renewcommand{\arraystretch}{1.2}
    \begin{array}{ll}
      &(m!)^{2}f'(e)f(e k_n) 
      \\ & \hspace{-0.4cm}
      =
       \left( m f^{(m)}(0) e^{m-1} + f^{(m+1)}(\xi_1) e^{m} \right) 
       \left(f^{(m)}(0) k_n^{m}e^{m} 
             + \frac{1}{m+1} f^{(m+1)}(\zeta_0) k_n^{m+1}e^{m+1} \right) 
      \\ & \hspace{-0.4cm}
      = \left[ m \left(f^{(m)}(0)\right)^2 k_n^m e^{2m-1} \right. 
      \\ & \hspace{.15cm}
      \left.
        + f^{(m)}(0)\left( f^{(m+1)}(\xi_1)k_n^m
        + f^{(m+1)}(\zeta_0)\frac{m}{m+1}k_n^{m+1} \right) e^{2m}
        + O(e^{2m+1}) \right], 
      \\ & 
      (m!)^{2}k_nf(e)f'(e k_n)
      \\ & \hspace{-0.4cm}
      =
       \left( f^{(m)}(0) e^{m} + \frac{1}{m+1} f^{(m+1)}(\xi_0) e^{m+1} \right)
       \left(m f^{(m)}(0) k_n^{m}e^{m-1} 
             + f^{(m+1)}(\zeta_1) k_n^{m+1}e^{m} \right) 
      \\ & \hspace{-0.4cm}
      = \left[m \left(f^{(m)}(0)\right)^2  k_n^{m}e^{2m-1} \right.
      \\ & \hspace{.15cm}
      \left.
      + f^{(m)}(0)\left( f^{(m+1)}(\zeta_1)k_n^{m+1} 
         + f^{(m+1)}(\xi_0)\frac{m}{m+1}k_n^{m} \right) e^{2m} 
      + O(e^{2m+1}) \right],
    \end{array}
  \end{displaymath}
   and
  \begin{displaymath}
    \renewcommand{\arraystretch}{1.2}
    \begin{array}{ll}
      &\left( f(e) - f(e k_n) \right)^2 
      \\ & \hspace{-0.4cm}
      =\frac{1}{(m!)^{2}}
      \left[\left( f^{(m)}(0) e^{m} 
                   + \frac{1}{m+1} f^{(m+1)}(\xi_0) e^{m+1} \right) \right.
      \\ & \hspace{1cm} 
      \left. 
        - \left(f^{(m)}(0) k_n^m e^{m} 
                + \frac{1}{m+1} f^{(m+1)}(\zeta_0) k_n^{m+1} e^{m+1}\right)
      \right]^2  
      \\ & \hspace{-0.4cm}
      = \frac{1}{(m!)^{2}}
        \left[ f^{(m)}(0)(1 - k_n^m)e^{m} 
              + \frac{1}{m+1}\left(f^{(m+1)}(\xi_0) 
                                - f^{(m+1)}(\zeta_0)k_n^{m+1} \right)e^{m+1}
        \right]^2 
      \\ & \hspace{-0.4cm}
      = \frac{1}{(m!)^{2}}
        \left[ \left(f^{(m)}(0)\right)^2(1 - k_n^m)^2e^{2m} 
              + O(e^{2m+1})\right],
    \end{array}
  \end{displaymath}
   where the $O(e^{2m+1})$ in the last step of the above three equations is 
   due to the fact that the absolute values of the coefficients 
   of higher-order terms all have constant upper bounds:
  \begin{displaymath}
    \renewcommand{\arraystretch}{1.2}
    \begin{array}{ll}
      &\left|f^{(m+1)}(\xi_1) \cdot 
       \frac{1}{m+1}f^{(m+1)}(\zeta_0)k_n^{m+1}\right|
       \le \frac{1}{m+1}\left(M_{f, \calB}^{(m+1)}\right)^2|k_n|^{m+1}
      \\ &\hspace{5.36cm}
       \le \frac{1}{m+1}\left(M_{f, \calB}^{(m+1)}\right)^2,\\
      &\left|\frac{1}{m+1}f^{(m+1)}(\xi_0) \cdot 
       f^{(m+1)}(\zeta_1)k_n^{m+1}\right|
       \le \frac{1}{m+1}\left(M_{f, \calB}^{(m+1)}\right)^2|k_n|^{m+1}
      \\ &\hspace{5.36cm}
       \le \frac{1}{m+1}\left(M_{f, \calB}^{(m+1)}\right)^2,\\
      &\left|2 \cdot f^{(m)}(0)(1 - k_n^m) \cdot \frac{1}{m+1} 
             \left(f^{(m+1)}(\xi_0) - f^{(m+1)}(\zeta_0)k_n^{m+1} \right)
       \right| 
      \\ &\hspace{5.36cm}
       \le 2 \cdot |f^{(m)}(0)|(1+|k_n|^m) \cdot 
            \frac{M_{f, \calB}^{(m+1)}}{m+1}
      \\ &\hspace{5.36cm}
       \le \frac{8 M_{f, \calB}^{(m+1)}|f^{(m)}(0)|}{m+1},\\
      &\left|\frac{1}{m+1}\left(f^{(m+1)}(\xi_0) - 
             f^{(m+1)}(\zeta_0)k_n^{m+1} \right)\right|^2
       \le \left|\frac{1}{m+1}M_{f, \calB}^{(m+1)}(1+|k_n|^{m+1})\right|^2\\
       &\hspace{6.1cm}\le \left(\frac{2M_{f, \calB}^{(m+1)}}{m+1}\right)^2.
    \end{array}
  \end{displaymath}
  Therefore,
  \begin{displaymath}
    \renewcommand{\arraystretch}{1.2}
    \begin{array}{ll}
      &f'(e)f(e k_{n})-k_{n}f(e)f'(e k_{n})
      \\ & \hspace{-0.43cm}
      = \frac{1}{(m!)^{2}}f^{(m)}(0)
      \left(f^{(m+1)}(\xi_1) - \frac{m}{m+1} f^{(m+1)}(\xi_0) 
      + \frac{m}{m+1}f^{(m+1)}(\zeta_0)k_n - f^{(m+1)}(\zeta_1)k_n \right)
      \\ &
      \ \cdot \ (1 - k_n)k_n^{m-1} e^{2m} + O(e^{2m+1}).
    \end{array}
  \end{displaymath}
  Substituting the expressions for
   $f'(e) f(e k_{n})-k_{n} f(e) f'(e k_{n})$
   and $\big( f(e) - f(e k_n) \big)^2$ into $g_n'(e)$,
   and using the conditions $f^{(m)}(0) \neq 0$,
   $e \neq 0$, and $k_n \leq r < 1$,
   we cancel the common factor $\frac{1}{(m!)^{2}}f^{(m)}(0)(1 - k_n)e^{2m}$
   from the numerator and denominator of $g_n'(e)$ to obtain
  \begin{displaymath}
    \renewcommand{\arraystretch}{1.2}
    \begin{array}{ll}
      g_n'(e)
      = \frac{\left(f^{(m+1)}(\xi_1) - \frac{m}{m+1} f^{(m+1)}(\xi_0) +
                     \frac{m}{m+1}f^{(m+1)}(\zeta_0)k_n - f^{(m+1)}(\zeta_1)k_n
              \right)
        k_n^{m-1} + O(e)}
        {f^{(m)}(0)(1-k_n^m)\left(\sum_{j=0}^{m-1} k_n^j\right)  + O(e)}.
    \end{array}
  \end{displaymath}
  Since 
   $ \left|f^{(m)}(0)(1-k_n^m)\left(\sum_{j=0}^{m-1} k_n^j\right)\right|
     > \left| f^{(m)}(0)\right|\cdot |1 - r^m| \cdot 1>0 $
   and
  \begin{displaymath}
    \renewcommand{\arraystretch}{1.2}
    \begin{array}{ll}
      \left| 
      \frac{ \left(f^{(m+1)}(\xi_1) - \frac{m}{m+1} f^{(m+1)}(\xi_0) 
                  + \frac{m}{m+1}f^{(m+1)}(\zeta_0)k_n 
                  - f^{(m+1)}(\zeta_1)k_n \right)k_n^{m-1}}
            {f^{(m)}(0)(1-k_n^m)\left(\sum_{j=0}^{m-1} k_n^j\right)} 
      \right| 
      < \frac{4 M_{f, \calB}^{(m+1)}}
            { \left| f^{(m)}(0) \right| \cdot (1 - r^m) },
    \end{array}
  \end{displaymath}
  Lemma~\ref{lem:epsLimit} shows that for $\epsilon = 1 > 0$,
  there exists $\gamma_1 \in(0,\delta)$ such that
  for all $e\in (0, \gamma_1)$, 
  \begin{displaymath}
    \renewcommand{\arraystretch}{1.2}
    \begin{array}{ll}
      &\left| g_n'(e) \right|
      < \left| 
        \frac{ \left(f^{(m+1)}(\xi_1) - \frac{m}{m+1} f^{(m+1)}(\xi_0) 
                  + \frac{m}{m+1}f^{(m+1)}(\zeta_0)k_n 
                  - f^{(m+1)}(\zeta_1)k_n \right)k_n^{m-1}}
            {f^{(m)}(0)(1-k_n^m)\left(\sum_{j=0}^{m-1} k_n^j\right)} 
        \right|  
      + \epsilon 
      \\ & \hspace{1.08cm}  
      < \frac{4 M_{f, \calB}^{(m+1)}}{(1-r^m)\left| f^{(m)}(0)\right|} 
      + \epsilon 
      = G, 
    \end{array}
  \end{displaymath}
  where the constant $G$ is as shown in \eqref{eq:GForm}.
\end{proof}

Finally, we combine the above conclusions to prove that 
 the secant method converges linearly near multiple roots of the function $f$.
\begin{theorem}
  \label{thm:EnLinearConvergence}
  Consider a function $f(x) \in \calC^{m+1}(\calB)$ with $m\ge 2$, 
   and $0$ is an $m$-fold root of $f$. 
  If the initial values $k_0$ and $e_0$ are chosen such that 
   $k_0 \in (0, 1)\cup (1, +\infty)$ and
   $e_0 > 0$ is sufficiently small,
   then the secant method error sequence $\{e_n\}$ converges linearly, 
   and its AEC $c = c_{m, 0}$,
   i.e.,
  \begin{equation*}
    \renewcommand{\arraystretch}{1.2}
    \begin{array}{ll} 
        \lim\limits_{n \to \infty} \frac{E_{n+1}}{E_{n}} 
      = \lim\limits_{n \to \infty} |k_n|   
      = c_{m, 0},
    \end{array}
  \end{equation*}
   where $c_{m, 0}$ is as defined in Lemma~\ref{lem:characterRoots0}.
\end{theorem}
\begin{proof}
  By Theorem~\ref{thm:mEnConvergence} and Lemma~\ref{lem:knRange},
   it follows that 
  \begin{align*}      
    \renewcommand{\arraystretch}{1.2}
    \begin{array}{ll}  
      \forall \epsilon > 0,\ 
      &\exists N_0 \st \forall \ell , n \geq N_0,\ 
       |e_\ell  - e_n| < \frac{\epsilon}{2G(m+1)};  \\
      &\exists N_1 \st \forall n \geq N_1,\ 
       0 < e_n < \min(\gamma_0, \gamma_1); \\
      &\exists N_2 \st \forall n \geq N_2,\ 
       |k_n| = k_n \in \left[\frac{2m-3}{2m}, r\right],
    \end{array}
  \end{align*}
   where $G$ is as described in Lemma~\ref{lem:gnPrimeBound}; 
   $\gamma_0, \gamma_1$ are as described in Lemma~\ref{lem:dnPrimeBound} 
   and Lemma~\ref{lem:gnPrimeBound}, respectively;
   $r \in \left(\frac{2m-3}{2m}, 1\right)$ is 
   as described in Lemma~\ref{lem:knRange}.
  We define $N_3 := \max(N_0, N_1, N_2)$,
   then for all $ n \geq N_3 $ and $ \ell > n$,
   $\range(k_\ell, k_n) \subseteq \left[\tfrac{2m-3}{2m}, r \right]$
   and $\range(e_{\ell}, e_{n}) \subseteq (0, \min(\gamma_0, \gamma_1))$.
  Applying Lemmas~\ref{lem:dnPrimeBound}~and~\ref{lem:gnPrimeBound}, 
   respectively, we get
  \begin{align}
     \forall \ell \geq N_3,\ 
    &\forall k \in \left[\tfrac{2m-3}{2m}, r\right],\ 
     |d_\ell '(k)| < \tfrac{m}{m+1}, \label{eq:dlPrimeBound}  \\
     \text{and} \quad
     \forall n \geq N_3,\ 
    &\forall e \in(0, \gamma_1),\ 
     |g_n'(e)| < G. \label{eq:gnPrimeBound}
  \end{align}
  Since $\frac{m}{m+1} < 1$, it holds that 
  \[
    \renewcommand{\arraystretch}{1.2}
    \begin{array}{ll}
      \exists N > N_3 \st 
      \forall n \geq N,\ 
      2 \left(\frac{m}{m+1}\right)^{n+1 - N_3} < \frac{\epsilon}{2}.
    \end{array}
  \]
  For any $ n \geq N > N_3 $ and $ \ell > n$,
   from \eqref{eq:fixKn}, we know
   $\hat{k}_{\ell, n} \in \range(k_\ell, k_n) \subseteq 
     \left[\frac{2m-3}{2m},r\right]$ 
   and $\hat{e}_{\ell, n} \in \range(e_{\ell}, e_{n}) \subseteq 
    (0, \gamma_1)$.
  Given the above inclusions, we derive
  \begin{align*}
    \renewcommand{\arraystretch}{1.2}
    \begin{array}{ll}
      & |k_{\ell +1} - k_{n+1}| 
      = |d_\ell '(\hat{k}_{\ell, n}) (k_{\ell} - k_{n}) +
        g_{n}'(\hat{e}_{\ell, n}) (e_{\ell} - e_{n}) |
      \\ & \hspace{2.13cm}
      \leq |d_\ell '(\hat{k}_{\ell, n}) (k_{\ell} - k_{n}) |
       + | g_{n}'(\hat{e}_{\ell, n}) (e_{\ell} - e_{n}) | 
      \\ & \hspace{2.13cm}
      < \frac{m}{m+1} |k_{\ell} - k_{n}| + G|e_{\ell} - e_{n}|,
    \end{array}
  \end{align*}
   where the second step follows from the triangle inequality,
   and the last from the fact that 
   $\hat{k}_{\ell, n} \in \left[ \frac{2m-3}{2m}, r \right]$ 
   and \eqref{eq:dlPrimeBound} imply 
   $|d_\ell '(\hat{k}_{\ell, n})| < \frac{m}{m+1}$, while
   $\hat{e}_{\ell, n} \in (0, \gamma_1)$ and  
   \eqref{eq:gnPrimeBound} yield $|g_n'(\hat{e}_{\ell, n})| < G$. 
  Applying induction to $|k_{\ell +1} - k_{n+1}| < 
   \frac{m}{m+1} |k_{\ell} - k_{n}| + G|e_{\ell} - e_{n}|$, we obtain
  \begin{displaymath}
    \renewcommand{\arraystretch}{1.2}
    \begin{array}{rl}
      &|k_{\ell +1} - k_{n+1}|
      \\ & \hspace{-0.43cm}
      < \left(\frac{m}{m+1} \right)^{n + 1 - N_3} |k_{\ell -n+N_3} - k_{N_3}| 
        + G\sum_{i=0}^{n-N_3} \left( \frac{m}{m+1} \right)^{i} 
        |e_{\ell -i} - e_{n-i}| 
      \\ & \hspace{-0.43cm}
      < \left(\frac{m}{m+1} \right)^{n + 1 - N_3} \cdot 2  
        + G \cdot \frac{1}{1 - \frac{m}{m+1}} \cdot \frac{\epsilon}{2G(m+1)} 
      < \frac{\epsilon}{2} +  \frac{\epsilon}{2} 
      = \epsilon.
    \end{array}
  \end{displaymath}
  Therefore, $\{k_n\}$ is a Cauchy sequence, 
   and it converges to some $c \in \left[ \frac{2m-3}{2m}, r \right]$.

  Next, we compute $d_{\ell}(k)$ on the domain
   $k \in \left[ \frac{2m-3}{2m}, r \right]$:
  \begin{displaymath}
    \renewcommand{\arraystretch}{1.7}
    \begin{array}{ll}
      &d_\ell (k)
      = \frac{f(e_{\ell }) - \frac{f(e_{\ell }k)}{k}}
             {f(e_{\ell }) - f(e_{\ell }k)} 
      = \frac{\left(\frac{f^{(m)}(0)}{m!}e_\ell ^m + O(e_\ell ^{m+1})\right) 
              - \left(\frac{f^{(m)}(0)}{m!}e_\ell ^mk^{m-1} 
                      + O((e_\ell  k)^{m+1})\right)}
             {\left(\frac{f^{(m)}(0)}{m!}e_\ell ^m + O(e_\ell ^{m+1})\right) 
              -\left(\frac{f^{(m)}(0)}{m!} e_\ell ^mk^{m} 
                      + O((e_\ell  k)^{m+1})\right)}
      \\ & \hspace{3.1cm}
      = \frac{(1 - k^{m-1})e_{\ell}^{m} + O(e_\ell^{m+1}) 
              + O((e_\ell k)^{m+1})}
             {(1 - k^m)e_{\ell}^{m} + O(e_\ell^{m+1} )+ O((e_\ell k)^{m+1})},
    \end{array}
  \end{displaymath}
   where the second step applies Taylor expansion of $f(e_\ell )$ 
   and $f(e_\ell  k)$ at $0$ using \eqref{eq:ftaylor},
   and the last follows from $ \frac{f^{(m)}(0)}{m!}\neq 0 $.
  Since $k \leq r < 1$ has a constant upper bound, the term
   $O((e_\ell k)^{m+1})$ can be absorbed into $O(e_\ell ^{m+1})$.
  Then, dividing both the numerator and denominator by $e_\ell^{m}$ yields
  \begin{equation}
    \label{eq:dlForm}
    \renewcommand{\arraystretch}{1.2}
    \begin{array}{ll}
      d_\ell (k) = \frac{(1 - k^{m-1}) + O(e_\ell )}{(1 - k^m) + O(e_\ell )}.
    \end{array}
  \end{equation}
  Define 
  \begin{equation}
    \label{eq:hform}
    \renewcommand{\arraystretch}{1.2}
    \begin{array}{ll}
      h_m(k) := \frac{1- k^{m - 1}}{1 - k^{m}} 
    \end{array}
  \end{equation}
   so that for $0 < k \leq r < 1$,
   $\left| 1 - k^m \right| \geq 1 - r^m > 0$ 
   and $|h_m(k)| = \frac{1 - k^{m-1}}{1 - k^m} < 1$. 
  By Lemma~\ref{lem:epsLimit}, for any $ \epsilon_d > 0$,
   there exists $\delta_d > 0$ such that 
   for all $k \in \left[\frac{2m-3}{2m}, r\right]$,
   provided that $|e_\ell| < \delta_d$, 
   we have $|d_\ell (k) - h_m(k)|  < \epsilon_d$. 
  Since the sequence $\{e_\ell \}$ converges to $0$, 
   there exists $N_d>N$ such that for all $\ell > N_d$, $|e_\ell| < \delta_d$. 
  In summary,
  \[
    \renewcommand{\arraystretch}{1.2}
    \begin{array}{ll}
      \forall \epsilon_d>0, \exists N_d>0 \st 
      \forall \ell > N_d, \forall k \in \left[\frac{2m-3}{2m}, r\right],\ 
      |d_\ell (k) - h_m(k)| < \epsilon_d,
    \end{array}
  \]
   which means that $d_\ell (k)$ converges uniformly to $h_m(k)$ on 
   $\left[ \frac{2m-3}{2m}, r \right]$.
  It holds that
  \[
      h_m(c) = \lim\limits_{k \to c} h_m(k) 
             = \lim\limits_{\ell \to \infty} h_m(k_\ell ) 
             = \lim\limits_{\ell \to \infty} d_\ell (k_\ell )  
             = \lim\limits_{\ell \to \infty} k_{\ell +1} = c,
  \] 
   where the first step follows from the continuity of $h_m(k)$,
   the second step from the assumption 
   $\lim\limits_{\ell \to\infty} k_\ell  = c \in \left[\frac{2m-3}{2m}, r\right]$,
   the third step from the uniform convergence of $d_\ell (k)$, 
   and the fourth step from \eqref{eq:dlklrelation}.
  Therefore, $c \in  \left[\frac{2m-3}{2m}, r \right] \subseteq (0, 1)$ 
   satisfies 
  \begin{equation}
    \label{eq:hFixPoint}
    \renewcommand{\arraystretch}{1.2}
    \begin{array}{ll}
        h_m(c) - c 
      = \frac{1 - c^{m-1} - c^{m}}
           {\sum_{j=0}^{m-1} c^{j}} 
      = -\frac{p_m(c)}{\sum_{j=0}^{m-1} c^{j}} 
      = 0
    \Longrightarrow p_m(c) = 0,
    \end{array}
  \end{equation}
   which, together with Lemma~\ref{lem:characterRoots0}, shows that 
   $c$ is the unique solution $c_{m, 0}$ of the equation
   $p_m(k) = 0$ in $(0, 1)$. 
  Thus, $\lim\limits_{n \to \infty} k_n = c_{m, 0}$, completing the proof.
\end{proof}


%% file: sec/sufficientConditionsMR.tex
Next, we consider the case where $e_0 = x_0$ and $e_1 = x_1$ 
 have opposite signs, i.e., $k_0 = \frac{e_1}{e_0} < 0$. 
By Theorem~\ref{thm:EnLinearConvergence}, we know 
 that for initial values $k_0 \in (0, 1)\cup (1, +\infty)$ and
 $e_0 > 0$ sufficiently small,
 the secant method error sequence $\{e_n\}$ converges linearly.
The following lemma shows 
 that $k_0 \in (-1, 0)$ together with sufficiently small $|e_0|$ 
 is a sufficient condition for Q-linear convergence of 
 the secant method error sequence $\{e_n\}$. 
We do not consider the case $k_0 = 0$, which corresponds 
 to premature termination of the secant method iteration
 as described in Hypothesis~\ref{hyp:nonTerminate}.
\begin{lemma}
  \label{lem:convergenceSufficientCondition}
  Consider a function $f(x) \in \calC^{m+1}(\calB)$ with $m \ge 2$, 
   and $0$ is an $m$-fold root of $f$.
  If the initial values $k_0$ and $e_0$ are chosen such that 
  $ k_0 \in (-1, 0) \cup (0, 1) \cup (1, +\infty) $
   and $|e_0|$ is sufficiently small, 
   then the error sequence $\{e_n\}$ converges linearly.
\end{lemma}
\begin{proof}
  When $k_0 \in (0, 1) \cup (1, +\infty)$, without loss of generality, 
   assume $e_0 > 0$; otherwise, consider the function $\tilde{f}(x) = f(-x)$. 
  By Theorem~\ref{thm:EnLinearConvergence}, if $e_0$ is sufficiently small,
   then the error sequence $\{e_n\}$ converges linearly.

  When $k_0 \in (-1, 0)$, we first show that $k_1 > 0$.
  Applying \eqref{eq:dlForm} and \eqref{eq:dlklrelation} 
   with $\ell = 0$, $k = k_0$, we obtain 
  \begin{equation}
    \label{eq:k1RecursiveRelation}
    \renewcommand{\arraystretch}{1.2}
    \begin{array}{ll}
    k_1 = d_0(k_0) 
        = \frac{1 - k_0^{m-1} + O(e_0)} 
               {1 - k_0^m + O(e_0)}.
    \end{array}
  \end{equation}
  Lemma~\ref{lem:epsLimit} shows that for
  \begin{displaymath}
    \renewcommand{\arraystretch}{1.2}
    \begin{array}{ll}
        \epsilon_0 = \frac{1 - k_0^{m-1}}{1 - k_0^m} > 0, \exists \delta_0 > 0
    \st \forall |x| < \delta_0, \ 
      \frac{1-k_0^{m-1} + O(x)}{1- k_0^m + O(x)} 
    > \frac{1 - k_0^{m-1}}{1 - k_0^m} - \epsilon_0 =  0,
    \end{array}
  \end{displaymath}
   which, together with sufficiently small $|e_0|$ 
   such that $|e_0| < \delta_0$, implies $k_{1} > 0$. 
    
  Next, we show that $k_1 \neq 1$.  
    
  When $m$ is even, we have $-1 < k_0^{m-1} < 0$ and $0 < k_0^m < 1$, which yields
  \begin{displaymath}
    \renewcommand{\arraystretch}{1.2}
    \begin{array}{ll}  
      1 < \frac{1}{1 - k_0^m} <  \frac{1 - k_0^{m-1}}{1 - k_0^m}.
    \end{array}
  \end{displaymath}
  By Lemma~\ref{lem:epsLimit}, for
  \begin{displaymath}
    \renewcommand{\arraystretch}{1.2}
    \begin{array}{ll}
      \epsilon_1 = \frac{-k_0^{m-1}}{1 - k_0^m} > 0, \exists \delta_1 > 0 
    \st  
      \forall |x| < \delta_1, \ 
      1 < \frac{1 - k_0^{m-1}}{1 - k_0^m} - \epsilon_1 
        < \frac{1-k_0^{m-1} + O(x)}{1- k_0^m + O(x)}.
    \end{array}
  \end{displaymath}
  Provided that $|e_0|$ is sufficiently small such that $|e_0| < \delta_1$,
   we conclude $1 < k_1$.

  When $m$ is odd, we have $-1 < k_0^{m} < 0$ and $0 < k_0^{m-1} < 1$, 
   which gives
  \begin{displaymath}
    \renewcommand{\arraystretch}{1.2}
    \begin{array}{ll} 
      \frac{1 - k_0^{m-1}}{1 - k_0^m} < \frac{1}{1 - k_0^m} < 1.
    \end{array}
  \end{displaymath}
  From Lemma~\ref{lem:epsLimit}, for
  \begin{displaymath}
    \renewcommand{\arraystretch}{1.2}
    \begin{array}{ll}
        \epsilon_2 = 1 - \frac{1 - k_0^{m-1}}{1 - k_0^m} > 0,
        \exists \delta_2 > 0
    \st \forall |x| < \delta_2, \
        \frac{1-k_0^{m-1} + O(x)}{1- k_0^m + O(x)}
      < \frac{1 - k_0^{m-1}}{1 - k_0^m} + \epsilon_2 = 1.
    \end{array}
  \end{displaymath}
    Combined with sufficiently small $|e_0|$ such that $|e_0| < \delta_2$, 
    we obtain $k_1 < 1$.
    
  In summary, when $k_0 \in (-1, 0)$ 
   and $|e_{0}|$ is chosen sufficiently small such that
   $|e_0| < \min(\delta_0, \delta_1, \delta_2)$, 
   we obtain $k_1 \in (0, 1) \cup (1, +\infty)$, 
   and $|e_1| = |k_0||e_0| = O(e_0)$ is also sufficiently small. 
  Considering $k_0 \leftarrow k_1, e_0 \leftarrow e_1$ as new initial values 
   for iteration of the secant method, we now have 
   $k_0 \in (0, 1) \cup (1, +\infty)$ and $|e_0|$ is sufficiently small. 
  Under these conditions, it was previously proven that
   $\{e_n\}$ converges linearly.
\end{proof}

In fact, the condition that $k_0 \in (-1, 0) \cup (0, 1) \cup (1, +\infty)$ 
 and $e_0$ satisfying $|e_0|$ is sufficiently small
 is only a sufficient condition for the Q-linear convergence of $\{e_n\}$.
Furthermore, depending on the parity of $m$, the admissible range of $k_0$ 
 can be extended: for sufficiently small $|e_0|$, 
 the sequence $\{e_n\}$ still converges linearly.

We first consider the case where $m$ is odd. 
For notational convenience, when discussing odd multiple roots, 
 we replace $m$ with $2m+1$, where $m\in\bbN^+$,
 that is, $0$ is a $(2m+1)$-fold root of $f \in \calC^{2m+2}(\calB)$.
The following theorem provides sufficient conditions 
 for the Q-linear convergence of $\{e_n\}$ with respect 
 to the values of $k_0$ and $e_0$ in the neighborhood of odd multiple roots.
\begin{theorem}
  \label{thm:knMoreRangeOdd}
  Consider a function $f(x) \in \calC^{2m+2}(\calB)$ with 
   $m\in\bbN^+$, and $0$ is a $(2m+1)$-fold root of $f$. 
  If the initial values $k_0$ and $e_0$ are chosen such that 
   $k_0 \notin \{-1, 0, 1\}$
   and $|e_0|$ is sufficiently small,
   then the error sequence $\{e_n\}$ converges linearly.
\end{theorem}
\begin{proof}
  By Lemma~\ref{lem:convergenceSufficientCondition},
   it is sufficient to consider the case $k_0 < -1$. 
  Applying \eqref{eq:k1RecursiveRelation} 
   with $m$ replaced by $2m+1$ yields
  \begin{displaymath}
    \renewcommand{\arraystretch}{1.2}
    \begin{array}{ll}
      k_1  = d_0(k_0) = \frac{1-k_0^{2m} + O(e_0)}{1 - k_0^{2m+1} + O(e_0)}.
    \end{array}
  \end{displaymath}
  When $k_0 < -1$, we have $1 - k_0^{2m} < 0 < 1 - k_0^{2m+1}$, which implies 
   $\frac{1-k_0^{2m}}{1 - k_0^{2m+1}} < 0$ and 
  \begin{displaymath}
    \renewcommand{\arraystretch}{1.2}
    \begin{array}{ll}
        \frac{1-k_0^{2m}}{1 - k_0^{2m+1}} 
      = \frac{1-|k_0|^{2m}}{1 + |k_0|^{2m+1}} 
      > \frac{1-|k_0|^{2m+1}}{1 + |k_0|^{2m+1}} 
      > -1.
    \end{array}
  \end{displaymath}
  This shows that $\frac{1 - k_0^{2m}}{1 - k_0^{2m+1}} \in (-1, 0)$.
  Then Lemma~\ref{lem:epsLimit} yields that for
  \begin{displaymath}
    \renewcommand{\arraystretch}{1.2}
    \begin{array}{ll}
      &\epsilon = \min\left(- \frac{1 - k_0^{2m}}{1 - k_0^{2m+1}}, 
                            1 + \frac{1-k_0^{2m}}{1 - k_0^{2m+1}}\right) > 0,
        \ \exists \delta_0 > 0  \\
  \st &\forall |x| < \delta_0, \ 
        -1 \le \frac{1-k_0^{2m}}{1 - k_0^{2m+1}} - \epsilon 
      < \frac{1-k_0^{2m} + O(x)}{1- k_0^{2m+1} + O(x)} 
      < \frac{1-k_0^{2m}}{1 - k_0^{2m+1}} + \epsilon 
      \le 0. 
    \end{array}
  \end{displaymath}
  Combined with the condition that $|e_0|$ is sufficiently small
   such that $|e_0| < \delta_0$, it follows that $-1 < k_1 < 0$. 
  Since $|e_1| = |k_0||e_0| = O(e_0)$ is also sufficiently small,
   by taking $k_0 \leftarrow k_1$ and  $e_0 \leftarrow e_1$
   as new initial values for the iteration of the secant method,
   the initial values $k_0, e_0$ satisfy the conditions of
   Lemma~\ref{lem:convergenceSufficientCondition}. 
  Consequently, the error sequence $\{e_n\}$ converges linearly.
\end{proof}

Next, we consider the case where $m$ is even. 
For notational convenience, when discussing even multiple roots, 
 we replace $m$ with $2m$, where $m\in\bbN^+$,
 i.e., $0$ is a $2m$-fold root of $f \in \calC^{2m+1}(\calB)$.
Note that, in this setting,
 it is not true that for arbitrary $k_0 \notin \{-1, 0, 1\}$, 
 the error sequence $\{e_n\}$ converges linearly
 whenever $|e_0|$ is sufficiently small.
To establish more general sufficient conditions for the convergence 
 of the secant method in this case, we first introduce the following 
 lemma and constants. 
\begin{lemma}
  \label{lem:characterRoots}
  For $m \geq 1$, replacing $m$ by $2m$ 
   in $p_m(k)$ as defined in \eqref{eq:characterEqm0}  
   yields the characteristic equation
  \begin{equation}
    \label{eq:characterEq0}
    p_{2m}(k) = k^{2m} + k^{2m-1} - 1 = 0.
  \end{equation}
  This equation has exactly two real solutions, 
   and one of them is $c_{2m, 0} \in (0, 1)$. 
  If we denote the other solution as $c_{2m, 1}$, 
   then $c_{2m, 1} \in (-2, -1)$.
  Furthermore, the characteristic equation 
  \begin{equation}
    \label{eq:characterEq1}
    {q}_{2m}(k) := p_{2m}(k) - 1 = k^{2m} + k^{2m-1} - 2 = 0
  \end{equation}
   has a unique solution $c_{2m, 2}$ on $[-2, c_{2m, 1})$. 
  In this case, it holds that 
  \begin{equation}
    \label{eq:tgValue}
    \forall x \in (-\infty, c_{2m, 2}), \ 
    \forall y \in (c_{2m, 2}, -1), \quad 
    {q}_{2m}(x) > {q}_{2m}(c_{2m, 2}) = 0 > {q}_{2m}(y).
  \end{equation} 
\end{lemma}
\begin{proof}
  Since $p_{2m}'(k)  = k^{2m-2}(2mk + 2m - 1)$ is negative 
   on $\left( -\infty, -\frac{2m-1}{2m} \right)$ 
   and positive on 
   $\left( -\frac{2m-1}{2m}, 0 \right) \cup \left( 0, +\infty \right)$, 
   $p_{2m}(k)$ is strictly monotonically decreasing 
   on $\left(-\infty, -\frac{2m-1}{2m}\right)$
   and strictly monotonically increasing on
   $\left(-\frac{2m-1}{2m}, 0 \right)$ 
   and $\left(0, +\infty\right)$.
  Note that $p_{2m}(0) = -1 < 0 < 1 = p_{2m}(1)$. 
  The intermediate value theorem 
   and Lemma~\ref{lem:characterRoots0}
   yield that \eqref{eq:characterEq0} has 
   a unique solution $c_{2m, 0}$ on $(0, 1)$. 
  Moreover, for 
   $k \in \left[ -\frac{2m-1}{2m}, 0 \right)$, $p_{2m}(k) < p_{2m}(0) < 0$, 
   i.e., \eqref{eq:characterEq0} has no real solution 
   on $\left[ -\frac{2m-1}{2m}, 0 \right)$;
   for $k \in \left[1, +\infty \right)$, $p_{2m}(k) \ge p_{2m}(1) > 0$, 
   i.e., \eqref{eq:characterEq0} has no real solution on $\left[1, +\infty\right)$.
  Similarly, from
  \begin{equation}
    \label{eq:gValue2}
         p_{2m}(-2) = 2^{2m} - 2^{2m-1} - 1 = 2^{2m-1} - 1 
    \geq 1 > 0 > -1 = p_{2m}(-1),
  \end{equation}
   we know that \eqref{eq:characterEq0} has a unique solution
   $c_{2m, 1} \in (-2, -1)$ on $\left( -\infty, -\frac{2m-1}{2m} \right)$. 

  For the characteristic equation \eqref{eq:characterEq1}, 
   since ${q}_{2m}'(k) = p_{2m}'(k) < 0$ on 
   $[-2, c_{2m, 1}) \subseteq \left( -\infty, -\frac{2m-1}{2m} \right)$,
   ${q}_{2m}(k)$ is strictly monotonically decreasing on $[-2, c_{2m, 1})$.
  When $m = 1$, \eqref{eq:characterEq1} becomes 
   $q_2(k) = k^2 + k - 2 = 0$, 
   which has a unique solution $c_{2, 2}=-2$ on $[-2, c_{2, 1})$.
  When $m > 1$, from \eqref{eq:gValue2}, we have 
  \[ 
      {q}_{2m}(-2) = p_{2m}(-2) - 1 \ge 2^{2m-1} - 2 
    > 0 > -1 = p_{2m}(c_{2m, 1})-1 = {q}_{2m}(c_{2m, 1}).
  \]
  Then, by the intermediate value theorem, equation \eqref{eq:characterEq1}
   has a solution $c_{2m, 2}$ on $(-2, c_{2m, 1})$.
  Combined with the fact that $q_{2m}(k)$ is
   strictly monotonically decreasing on $(-\infty, -1)$,
   this solution is unique, and \eqref{eq:tgValue} holds.
\end{proof}

Using the constants introduced in Lemma~\ref{lem:characterRoots}, 
 the following lemma extends the sufficient conditions 
 for the Q-linear convergence of $\{e_n\}$ with respect 
 to the values of $k_0$ and $e_0$
 in the neighborhood of even multiple roots.
\begin{lemma}
  \label{lem:knMoreRangeEven}
  Consider a function $f(x) \in \calC^{2m+1}(\calB)$ with
   $m\in\bbN^+$, and $0$ is a $2m$-fold root of $f$. 
  If the initial values $k_0$ and $e_0$ are chosen such that 
   $k_0 \in (-\infty, c_{2m, 2}) \cup (-1, 0) \cup (0, 1) \cup (1, +\infty)$
   and $|e_0|$ is sufficiently small,
   then the error sequence $\{e_n\}$ converges linearly.
\end{lemma}
\begin{proof}
  By Lemma~\ref{lem:convergenceSufficientCondition}, 
   it suffices to consider the case $k_0 \in (-\infty, c_{2m, 2})$. 
  Applying \eqref{eq:k1RecursiveRelation} 
   with $m$ replaced by $2m$ yields
  \begin{displaymath}
    \renewcommand{\arraystretch}{1.2}
    \begin{array}{ll}
      k_1  = d_0(k_0) = \frac{1-k_0^{2m-1} + O(e_0)}{1 - k_0^{2m} + O(e_0)}.
    \end{array}
  \end{displaymath}
  From $k_0 < c_{2m, 2} < -1$, it follows that
   $1 - k_0^{2m-1} > 0 > 1 - k_0^{2m}$, 
   which implies $\frac{1-k_0^{2m-1}}{1 - k_0^{2m}} < 0$.
  Moreover, \eqref{eq:tgValue} gives
  \begin{displaymath}
    \renewcommand{\arraystretch}{1.2}
    \begin{array}{ll}
        \frac{1 - k_0^{2m-1}}{1 - k_0^{2m}} + 1 
      = \frac{2 - k_0^{2m-1} -  k_0^{2m}}{1 - k_0^{2m}} 
      = \frac{{q}_{2m}(k_0)}{k_0^{2m} - 1}
      > \frac{{q}_{2m}(c_{2m, 2})}{k_0^{2m} - 1} = 0. 
    \end{array}
  \end{displaymath}
  Hence $\frac{1 - k_0^{2m-1}}{1 - k_0^{2m}} \in (-1, 0)$.
  Then Lemma~\ref{lem:epsLimit} shows that for 
  \begin{displaymath}
    \renewcommand{\arraystretch}{1.2}
    \begin{array}{ll}
     &\epsilon = \min\left(-\frac{1-k_0^{2m-1}}{1 - k_0^{2m}}, 
                           1 +\frac{1-k_0^{2m-1}}{1 - k_0^{2m}}\right) > 0,\ 
      \exists \delta_0 > 0  \\
 \st &\forall |x| < \delta_0, \ 
      -1 \le \frac{1 - k_0^{2m-1}}{1 - k_0^{2m}} - \epsilon 
          < \frac{1-k_0^{2m-1} + O(x)}{1- k_0^{2m} + O(x)} 
          < \frac{1 - k_0^{2m-1}}{1 - k_0^{2m}} + \epsilon
          \le 0. 
    \end{array}
  \end{displaymath}
  Combined with the fact that $|e_0|$ is sufficiently small 
   such that $|e_0| < \delta_0$,
   we have $-1 < k_1 < 0$. 
  Since $|e_1| = |k_0||e_0| = O(e_0)$ is also sufficiently small, 
   we can consider taking $k_0 \leftarrow k_1$ and $e_0 \leftarrow e_1$ 
   as new initial values for the iteration of the secant method. 
  Then $k_0, e_0$ satisfy the conditions 
   of Lemma~\ref{lem:convergenceSufficientCondition}, 
   and thus the error sequence $\{e_n\}$ converges linearly.
\end{proof}

However, we can appropriately construct
 functions and initial values such that
 the error sequence $\{e_n\}$ of the secant method does not converge.

If, before the first element $k_\ell$ in the sequence $\{k_n\}$ 
 satisfies the convergence condition 
 $k_\ell \in (-\infty, c_{2m, 2}) \cup (-1, 0) \cup (0, 1) \cup (1, +\infty)$ 
 of Lemma~\ref{lem:knMoreRangeEven},
 there exists $i < \ell$ such that $k_i \in \{-1, c_{2m, 1}, c_{2m, 2}\}$. 
Then, for the function $f(x)=x^{2m}$ with any $e_0 \neq 0$, 
 the error sequence $\{e_n\}$ does not converge, 
 as demonstrated in the following examples.
\begin{example}
  \label{exam:k0eqc1}
  Consider the function $f(x) = x^{2m}$ with $m \in \bbN^+$, 
   and assume $k_{n-1} \neq \pm 1$. 
  From the recurrence relation \eqref{eq:knRecm}, we obtain 
  \begin{equation}
    \label{eq:knRec2m}
    \renewcommand{\arraystretch}{1.2}
    \begin{array}{l}
      k_{n} = \frac{f(e_{n-1}) - \frac{f(e_{n-1}k_{n-1})}{k_{n-1}}}
                   {f(e_{n-1}) - f(e_{n-1}k_{n-1})} 
            = \frac{1 - k_{n-1}^{2m-1}}{1 - k_{n-1}^{2m}} 
            = h_{2m}(k_{n-1}),
    \end{array}
  \end{equation}
   where $h_{2m}(k)$ is defined as \eqref{eq:hform}.
  Note that from \eqref{eq:hFixPoint}, the fixed points 
   of $h_{2m}(k)$ satisfy the characteristic equation $p_{2m}(k) = 0$. 
  Then by Lemma \ref{lem:characterRoots}, the two fixed points 
   of $h_{2m}(k)$ are $c_{2m, 0}$ and $c_{2m, 1}$.
   
  Taking $k_0 = c_{2m, 1} \in (-2, -1)$ and any $e_0 \neq 0$,
   which, together with iteration \eqref{eq:knRec2m},
   yields that for all
   $ n\in \bbN$, $k_{n} = k_0 = c_{2m, 1}$. 
  Because $|c_{2m, 1}| > 1$, 
   it follows that $|e_n| = |c_{2m, 1}|^{n} |e_0| \to +\infty$ 
   as $n \to \infty$. 
  Hence, the error sequence $\{e_n\}$ diverges. 
\end{example}

\begin{example}
  \label{exam:k0eqminus2}
  Consider the function $f(x) = x^{2m}$ with $m \in \bbN^+$. 
  If we take $k_0 = -1$ and $e_0 \neq 0$, 
   then $e_1 = -e_0$, and by formula \eqref{eq:secantMethod}, 
   we have $e_2 = \infty$.
  
  If we take $k_0 = c_{2m, 2}$ and $e_0 \neq 0$, 
   then by substituting 
   ${q}_{2m}(c_{2m, 2}) = c_{2m, 2}^{2m} + c_{2m, 2}^{2m-1} - 2 = 0$ 
   into formula \eqref{eq:knRec2m}, we can compute 
  \begin{displaymath}
    \renewcommand{\arraystretch}{1.2}
    \begin{array}{l}
      k_1 = h_{2m}(k_0) = \frac{1 - c_{2m, 2}^{2m-1}}{1 - c_{2m, 2}^{2m}} 
          = \frac{c_{2m, 2}^{2m} - 1}{1 - c_{2m, 2}^{2m}}
          = -1,
    \end{array}
  \end{displaymath}
   so $e_2 = k_1 e_1 = -e_1$, and by formula \eqref{eq:secantMethod}, 
   we have $e_3 = \infty$.
    
  In these cases, for any $e_0 \neq 0$, the error sequence $\{e_n\}$ diverges. 
\end{example}

To pave the way for establishing more general sufficient conditions 
 for the convergence of the secant method for general functions 
 $f(x) \in \calC^{2m+1}(\calB)$, we first analyze the prototypical case 
 $f(x) = x^{2m}$ with $m \in \bbN^+$.
\begin{lemma}
  \label{lem:fx2mConvergence}
  Consider the function $f(x) = x^{2m}$ with $m \in \bbN^+$.
  For given initial values $k_0 \neq 1$ and $e_0 \neq 0$,
   let $\{k_n\}$ be the sequence obtained by applying the secant method. 
  Let $\calB_{k, 2m} := (c_{2m, 2}, c_{2m, 1}) \cup (c_{2m, 1}, -1)$,
  where the constants $c_{2m, 0}, c_{2m, 1}$, and $c_{2m, 2}$ 
   are defined as in Lemma~\ref{lem:characterRoots}.
  Then 
  \begin{enumerate}[(i)]
    \item there exists an index $\ell \geq 0$ such that 
      $k_\ell \notin \calB_{k, 2m}$;
    \item Let $\ell$ be the smallest such index. Then the error sequence
      $\{e_n\}$ converges linearly if and only if 
      $k_\ell \notin \{-1, 0, 1, c_{2m, 1}, c_{2m, 2}\}$.
  \end{enumerate} 
\end{lemma}
\begin{proof}
  The recurrence for $\{k_n\}$ is given by 
   $k_{n} = h_{2m}(k_{n-1})$ as derived in \eqref{eq:knRec2m}.
  Taking the derivative of $h_{2m}(k)$ yields 
  \begin{equation}
    \label{eq:hprimeform}
    \renewcommand{\arraystretch}{1.2}
    \begin{array}{ll}
      h_{2m}'(k) =  \frac{2mk^{2m-1} - (2m-1) k^{2m-2} - k^{4m-2}}
                    {(1 - k^{2m})^2}.
    \end{array}
  \end{equation}
    
  We first prove that there exists $k_\ell \notin \calB_{k, 2m}$ 
   in the sequence $\{k_n\}$.
  Suppose, for contradiction, 
   that for all $\ell \in \bbN$, $k_\ell \in \calB_{k, 2m}$.
  Then, by the definition of $\calB_{k, 2m}$ and Example \ref{exam:k0eqc1},
   $k_\ell$ cannot be either of the two fixed points
   $c_{2m, 0}$ and $c_{2m, 1}$ of $h_{2m}(k)$.
  Hence $h_{2m}(k_{\ell}) \neq k_{\ell}$, which, 
   combined with \eqref{eq:knRec2m}, implies
   $k_{\ell+1} \neq k_\ell$.
  
  For $m = 1$, we have 
   $h_{2}(k) = \frac{1}{1 + k}$, $c_{2, 1}=-\frac{1 + \sqrt{5}}{2}$, 
   and $c_{2, 2}=-2$.
  If $k_\ell \in \left(-\frac{3}{2}, -1\right)$, then 
   $k_{\ell + 1} = h_{2}(k_\ell) < -2$, 
   which contradicts $k_{\ell +1} \in \calB_{k, 2m}$.
  If $k_\ell \in \left(-2,-\frac{5}{3}\right)$, then 
   $k_{\ell + 1} = h_{2}(k_\ell) \in \left(-\frac{3}{2}, -1\right)$ gives 
   $k_{\ell + 2} = h_{2}(k_{\ell + 1}) < -2$, 
   which contradicts $k_{\ell + 2} \in \calB_{k, 2m}$.
  Hence for all $\ell \in \bbN$, 
   $k_\ell \in \left[-\frac{5}{3}, c_{2, 1}\right) 
          \cup \left(c_{2, 1}, -\frac{3}{2}\right]$. 
  Note that for $k \in \left[ -\frac{5}{3}, -\frac{3}{2} \right]$, 
   $|h_{2}'(k)| = \frac{1}{(1+k)^{2}} \geq \frac{9}{4} > 2$. 
  Theorem~\ref{thm:meanValue} and $k_{\ell}\neq k_{\ell-1}$ show that 
   there exists 
   $\hat{k}_{\ell, \ell - 1} \in \range(k_\ell ,k_{\ell - 1}) 
    \subseteq \left[-\frac{5}{3}, -\frac{3}{2}\right]$ such that
  \[
      |k_{\ell +1} - k_{\ell }| = |h_{2}(k_{\ell }) - h_{2}(k_{\ell -1})| 
    = |h_{2}'(\hat{k}_{\ell, \ell - 1} )(k_\ell -k_{\ell -1})|
    > 2 |k_\ell  - k_{\ell -1}|,
  \]
  where the first step follows from \eqref{eq:knRec2m}.
  By induction, we obtain $|k_{\ell +1} - k_{\ell }| > 2^{\ell} |k_1 - k_0|$.
  Therefore, when $\ell$ is sufficiently large such that
   $2^{\ell}|k_{1}-k_{0}| > -1 - c_{2, 2}$, it holds that
   $ |k_{\ell +1} - k_{\ell }| 
   > -1 - c_{2, 2} = \sup_{k_x,k_y \in \calB_{k, 2m}} |k_x - k_y|$, 
   which contradicts $k_{\ell }, k_{\ell +1} \in \calB_{k, 2m}$. 

  For $m > 1$, from \eqref{eq:tgValue}, we have 
  \[
    \forall k \in (c_{2m, 2}, -1), \quad
    {q}_{2m}(k) = k^{2m} + k^{2m-1} - 2 
                = |k|^{2m} - |k|^{2m-1} - 2  < 0.
  \] 
  Hence, \eqref{eq:hprimeform} yields that for any $k \in (c_{2m, 2}, -1)$, 
  \begin{displaymath}
    \renewcommand{\arraystretch}{1.2}
    \begin{array}{ll}
      & \hspace{0.3cm}
         (1 - k^{2m})^2(|h_{2m}'(k)| - 1) \\
      &= 2m|k|^{2m-1} + (2m-1)|k|^{2m-2} + |k|^{4m-2} 
         - 1 + 2|k|^{2m} - |k|^{4m} \\
      &= 2m|k|^{2m-1} + (2m-1)|k|^{2m-2} + |k|^{4m-2} 
         - 1 + |k|^{2m}(2 - |k|^{2m}) \\
      &> 2m|k|^{2m-1} + (2m-1)|k|^{2m-2} + |k|^{4m-2} 
         - 1 + |k|^{2m}(-|k|^{2m-1})  \\
      &= 2m|k|^{2m-1} + (2m-1)|k|^{2m-2} 
         - 1 + |k|^{2m-1}(|k|^{2m-1} - |k|^{2m}) \\
      &> 2m|k|^{2m-1} + (2m-1)|k|^{2m-2} - 1 - 2|k|^{2m-1} \\
      &= 2(m-1)|k|^{2m-1} + (2m-1)|k|^{2m-2} - 1 
      \\ & 
      > 0+ 3|k|^{2m-2} - 1 > 1,
    \end{array}
  \end{displaymath}
   where the first step follows from the fact that
   $2mk^{2m-1}$, $-(2m-1)k^{2m-2}$ and $-k^{4m-2}$ are negative,
   and the third and fifth steps from $|k|^{2m} - |k|^{2m-1} - 2 < 0$.
  This shows that
  \[
    \renewcommand{\arraystretch}{1.2}
    \begin{array}{ll}
      \forall k \in (c_{2m, 2},-1),\quad
      |h_{2m}'(k)| > 1 + \frac{1}{(1 - k^{2m})^2} 
                   > C := 1 + \frac{1}{(2^{2m} - 1)^2} > 1,
    \end{array}
  \]
   where the second step follows from $k\in(c_{2m, 2},-1) \subseteq (-2, -1)$.
  Then Theorem~\ref{thm:meanValue} and $k_{\ell}\neq k_{\ell-1}$ imply that
   there exists 
   $\hat{k}_{\ell , \ell -1} \in \range(k_\ell ,k_{\ell - 1}) 
                             \subseteq \left(c_{2m, 2}, -1\right)$ such that
  \[
      |k_{\ell +1} - k_{\ell }| = |h_{2m}(k_{\ell }) - h_{2m}(k_{\ell -1})| 
    = |h_{2m}'(\hat{k}_{\ell , \ell -1})(k_\ell -k_{\ell -1})|
    > C |k_\ell  - k_{\ell -1}|,
  \]
   where the first step follows from \eqref{eq:knRec2m}. 
  By induction, we obtain $|k_{\ell +1} - k_{\ell }| > C^{\ell} |k_1 - k_0|$.
  Therefore, when $\ell$ is sufficiently large such that
   $C^{\ell}|k_{1} - k_{0}| > -1 - c_{2m, 2}$, it holds that 
   $|k_{\ell +1} - k_{\ell }| > -1 - c_{2m, 2} 
                              = \sup_{k_x, k_y \in \calB_{k, 2m}} |k_x - k_y|$, 
   which contradicts $k_{\ell }, k_{\ell+1} \in \calB_{k, 2m}$. 

  In summary, 
   there exists $k_\ell \notin \calB_{k, 2m}$ in the sequence $\{k_n\}$.   

  Next, let $k_\ell$ denote the first value in $\{k_n\}$ 
   that belongs to $\bbR \setminus \calB_{k, 2m}$.
  We prove that $k_\ell \notin \{-1, 0, 1, c_{2m, 1}, c_{2m, 2}\}$ 
   is a necessary and sufficient condition for Q-linear convergence 
   of the error sequence $\{e_n\}$.
    
  Necessity: 
  Examples~\ref{exam:k0eqc1} and~\ref{exam:k0eqminus2} have verified 
   that when $k_\ell \in \{-1, c_{2m, 1}, c_{2m, 2}\}$, 
   the error sequence $\{e_{n}\}$ diverges.
  If $k_{\ell} = 1$, then $e_{\ell} = e_{\ell+1}$, 
   the secant method formula \eqref{eq:secantMethod} is not well-defined.
  The case $k_\ell = 0$ is excluded by Hypothesis~\ref{hyp:nonTerminate}.

  Sufficiency: 
  For any $k \in (-\infty, c_{2m, 2})$, since $c_{2m, 2} < -1$, 
   we have $k^{2m} - 1 > 0$ and $h_{2m}(k) < 0$. 
  Then \eqref{eq:tgValue} gives
  \begin{displaymath}
    \renewcommand{\arraystretch}{1.2}
    \begin{array}{ll}
      h_{2m}(k)  = \frac{1- k^{2m-1}}{1 - k^{2m}} 
            = \frac{2- k^{2m-1} - k^{2m}}{1 - k^{2m}} - 1 
            = \frac{q_{2m}(k)}{k^{2m} - 1} -1 
            > \frac{q_{2m}(c_{2m, 2})}{k^{2m} - 1} - 1 = -1. 
    \end{array}
  \end{displaymath}
  Moreover, when $k \notin \{-1, 0, 1\}$, algebraic computation yields 
  \begin{displaymath}
    \renewcommand{\arraystretch}{1.2}
    \begin{array}{ll}
      h_{2m}(k) = \frac{1-k^{2m-1}}{1 - k^{2m}} 
           = \frac{\sum_{j=0}^{2m-2} k^j}{\sum_{j=0}^{2m-1} k^j} 
           = 1 - \frac{k^{2m-1}}{\sum_{j=0}^{2m-1} k^j} 
           = 1 - \frac{1}{\sum_{j=0}^{2m-1} \frac{1}{k^j}}
           = 1 - \frac{1 - \frac{1}{k}}{1 - \frac{1}{k^{2m}}}. 
    \end{array}
  \end{displaymath}
  From the above, we can verify that 
  \begin{displaymath}
    \renewcommand{\arraystretch}{1.2}
    \begin{array}{ll}
      &\forall k \in (-\infty, c_{2m, 2}),\ h_{2m}(k) \in (-1, 0), 
       \hspace{0.58cm} 
       \forall k \in (-1, 0),\ h_{2m}(k) \in (1, +\infty), \\
      &\forall k \in (1, +\infty),\ h_{2m}(k) \in (0, 1),    
       \hspace{1.48cm} 
       \forall k \in (0, 1),\ h_{2m}(k) \in \left(\frac{2m-1}{2m}, 1\right), \\
      &\forall k \in \left(\frac{2m-1}{2m}, 1\right),\ 
       h_{2m}(k) \in \left(\frac{2m-1}{2m}, a \right) 
            \subseteq \left[\frac{2m-1}{2m}, a \right],
    \end{array}
  \end{displaymath}
  where $a := 1 - \frac{(2m-1)^{2m-1}}{2m^{2m} - (2m-1)^{2m}} 
           \in \left(\frac{2m - 1}{2m}, 1 \right)$.
  Hence from $k_\ell \notin \{-1, 0, 1, c_{2m, 1}, c_{2m, 2}\}$ 
   and $k_\ell \notin \calB_{k, 2m}$, i.e., 
   $k_\ell \in(-\infty, c_{2m, 2}) \cup (-1, 0) \cup (0, 1) \cup (1, +\infty)$,
   it follows that
  \begin{displaymath}
    \renewcommand{\arraystretch}{1.2}
    \begin{array}{ll}
      &k_{\ell+1} = h_{2m}(k_\ell) \in (-1, 0) \cup (0, 1) \cup (1, +\infty),\\
      &k_{\ell+2} = h_{2m}(k_{\ell+1}) \in (0, 1) \cup (1, +\infty),\\
      &k_{\ell+3} = h_{2m}(k_{\ell+2}) \in (0, 1), \\
      &k_{\ell+4} = h_{2m}(k_{\ell+3}) \in \left(\frac{2m-1}{2m}, 1 \right), \\
      &k_{\ell+5} = h_{2m}(k_{\ell+4}) \in \left(\frac{2m-1}{2m}, a \right) 
                                    \subseteq \left[\frac{2m-1}{2m}, a \right].
    \end{array}
  \end{displaymath}
  By induction, for any $N > 0$, 
   $k_{N + \ell + 4}\in \left[\frac{2m-1}{2m}, a \right]$.
  Since $a \in \left(\frac{2m-1}{2m}, 1 \right)$ implies 
   $a \in \left(\frac{4m-3}{4m}, 1\right)$, 
   replacing $m$ and $r$ with $2m$ and $a$, respectively, 
   in $b(k)$ as defined in \eqref{eq:bkForm} 
   yields $b(k) = h_{2m}'(k)$.
  Then $\dom(h_{2m}') = \left[\frac{2m-1}{2m}, a\right] 
                 \subseteq \left[\frac{4m-3}{4m}, a\right] = \dom(b) $.
  Therefore, from \eqref{eq:bkRelease} and $b(k) < 0$, we obtain
  \begin{displaymath}
    \renewcommand{\arraystretch}{1.2}
    \begin{array}{ll}
      \forall k \in \left[\frac{2m-1}{2m}, a\right], \quad 
      |h_{2m}'(k)| < \frac{2m-1}{2m}.
    \end{array}
  \end{displaymath}

  For any $ N > 0$ and $ L > \ell + 5$, 
   if there exists an integer $\hat{i}$ 
   with $0 < \hat{i} \leq L - \ell - 5$ such that
   $|k_{N + L - \hat{i}} - k_{L - \hat{i}}| = 0$,
   then \eqref{eq:knRec2m} means that
   $|k_{N + L - \hat{i} + 1} - k_{L - \hat{i} + 1}| 
   = |h_{2m}(k_{N + L - \hat{i}}) - h_{2m}(k_{L - \hat{i}})| = 0$.
  By induction, it holds that for all $\hat{L} \ge L - \hat{i}$, 
   $|k_{N + \hat{L}} - k_{\hat{L}}| = 0$,
   and hence for any $N > 0$, $\lim\limits_{L\to\infty}|k_{N+L}-k_{L}| = 0$.
  Now assume that for all $0 < i \leq L - \ell - 5$, 
   $|k_{N + L - i} - k_{L - i}| \neq 0$.
  Then Theorem~\ref{thm:meanValue} yields
  \begin{displaymath}
    \renewcommand{\arraystretch}{1.2}
    \begin{array}{ll}
      & \hspace{0.2cm} \exists \hat{k}_{N + L - i, L - i} 
       \in \range(k_{N + L - i}, k_{L - i}) 
       \subseteq \left[\frac{2m-1}{2m}, a\right] \\
       \st
      &\hspace{0.35cm}  |k_{N + L - i + 1} - k_{L - i + 1}| 
       = |h_{2m}(k_{N + L - i}) - h_{2m}(k_{L - i})| \\
      &= |h_{2m}'(\hat{k}_{N + L - i, L - i})(k_{N + L - i} - k_{L - i})| 
      \leq \frac{2m-1}{2m}|k_{N + L - i} - k_{L - i}|, 
    \end{array}
  \end{displaymath}
   where the first step follows from \eqref{eq:knRec2m}.
  Therefore, an induction implies that
   $|k_{N + L} - k_{L}| \leq \left( \frac{2m-1}{2m}\right)^{L - \ell - 5}
                             |k_{N + \ell + 5} - k_{\ell  + 5}|$,
   which shows                          
  \begin{displaymath}
    \renewcommand{\arraystretch}{1.2}
    \begin{array}{rl}
      &\forall N > 0, \quad
          \lim\limits_{L \to \infty} |k_{N + L} - k_{L}| 
      \le \lim\limits_{L \to \infty}\left(\frac{2m-1}{2m}\right)^{L - \ell - 5}
                              |k_{N + \ell + 5} - k_{\ell  + 5}| 
      \\ & \hspace{4.45cm}
      \le \lim\limits_{L \to \infty}\left(\frac{2m-1}{2m}\right)^{L - \ell - 5} 
                                    \cdot 2 
      = 0.
    \end{array}
  \end{displaymath} 
  This means that $\{k_n\}$ is a Cauchy sequence, 
   and it converges to some $\hat{c} \in \left[ \frac{2m-1}{2m}, a \right]$. 

  As in \eqref{eq:hFixPoint}, $\hat{c}$ must be a fixed point of $h_{2m}(k)$.
  Example~\ref{exam:k0eqc1} shows that $h_{2m}(k)$ has 
   a unique fixed point $c_{2m, 0} \in \left[\frac{2m-1}{2m}, a \right]$,
   which forces that $\{k_n\}$ converges to $\hat{c} = c_{2m, 0} \in (0, 1)$.
  Hence for any $ e_0 \neq 0$, 
   the error sequence $\{e_n\}$ converges linearly. 
\end{proof}

The result of Lemma~\ref{lem:fx2mConvergence} can be generalized to establish
 a sufficient condition for Q-linear convergence 
 of the error sequence $\{e_n\}$ when $0$ is a $2m$-fold root 
 of a more general function $f(x) \in \calC^{2m+1}(\calB)$.
\begin{theorem}
  \label{thm:fx2mCommConvergence}
  Consider a function $f(x) \in \calC^{2m+1}(\calB)$ with 
   $m \in \bbN^+$, and $0$ is a $2m$-fold root of $f$. 
  For given initial values $k_0, e_0$, define
   $\tilde{k}_0 := k_0$, and let the sequence $\{\tilde{k}_n\}$ be defined 
   by the iteration $\tilde{k}_{n+1} := h_{2m}(\tilde{k}_n)$, 
   where $h_{2m}$ is defined in \eqref{eq:hform}. 
  If the sequence $\{\tilde{k}_n\}$ is such that 
   the first term $\tilde{k}_\ell$ that satisfies
   $\tilde{k}_\ell \notin \calB_{k, 2m}$ 
   also satisfies 
   $\tilde{k}_\ell \notin \{-1, 0, 1, c_{2m, 1}, c_{2m, 2}\}$, 
   then for sufficiently small $|e_0|$, 
   the error sequence $\{e_n\}$ converges linearly.
\end{theorem}
\begin{proof}
  By Lemma~\ref{lem:fx2mConvergence}, there exists $N \geq 0$ such that 
   for all $ n < N, \tilde{k}_n \in \calB_{k, 2m} $ 
   and $\tilde{k}_N \notin \calB_{k, 2m}$.
    
  For $N = 0$, we have $k_0 = \tilde{k}_0 \notin \calB_{k, 2m}$ 
   and $k_0 \notin \{-1, 0, 1, c_{2m, 1}, c_{2m, 2}\}$, that is,
   $k_{0} \in (-\infty, c_{2m, 2}) \cup (-1, 0) \cup (0, 1) 
              \cup (1, +\infty)$.
  Lemma~\ref{lem:knMoreRangeEven} shows that 
   when $|e_0|$ is sufficiently small, 
   the error sequence $\{e_n\}$ converges linearly.
  
  For $N > 0$, we prove by induction 
   that
  \[
    \forall n \leq N,\forall \epsilon_n > 0,\ 
    \exists \delta_n > 0 \st \forall |e_0| < \delta_n,\
     |k_{n} - \tilde{k}_n| < \epsilon_n.
   \]
  When $n = 0$, $k_0 = \tilde{k}_0$,
   hence, $|k_{0}-\tilde{k}_{0}|<\epsilon_{0}$ always holds.
  Now assume that the statement holds for $0, 1, \ldots, n < N$. 
  Then we prove that the statement holds for $n + 1$, i.e.,
  \[
      \forall \epsilon_{n+1} > 0, \exists \delta_{n+1} > 0 
  \st \forall |e_0| < \delta_{n+1},\ |k_{n+1} - \tilde{k}_{n+1}| 
                    < \epsilon_{n+1}.
  \]

  Since for $\ell \leq n < N$, 
   $\tilde{k}_\ell \in \calB_{k, 2m} 
                   \subseteq (c_{2m, 2}, -1)$, 
   by the inductive hypothesis, for
  \begin{equation}
    \label{eq:deltaN+10}
      \epsilon_\ell 
    = \frac{1}{2}\min(|\tilde{k}_\ell  - c_{2m, 2}|, 
                      |\tilde{k}_\ell  + 1|) 
    > 0,
      \exists \delta_\ell  > 0, 
    \st \forall |e_0| < \delta_{\ell },\ 
            |k_{\ell } - \tilde{k}_\ell | < \epsilon_\ell.
  \end{equation}
  Denote $\delta_{n+1, 0} := \min_{\ell=0}^{n} (\delta_\ell) > 0$.
  Then when $|e_0| < \delta_{n+1, 0}$,
   it follows that for all
   $\ell\leq n$, $|k_{\ell } - \tilde{k}_\ell | < \epsilon_\ell$. 
  Combined with
   $\epsilon_\ell \le \frac{1}{2}|\tilde{k}_\ell  + 1|$ 
   and $\tilde{k}_\ell  < -1$, we obtain
  \begin{equation}
    \label{eq:knInTildeknDomain1} 
    \renewcommand{\arraystretch}{1.2}
    \begin{array}{ll}
      &k_\ell  
       < \tilde{k}_\ell  + \epsilon_\ell  
       \le \tilde{k}_\ell + \frac{1}{2}|\tilde{k}_\ell  + 1|
       = \tilde{k}_\ell  + \frac{1}{2}(-1-\tilde{k}_\ell) \\
      & \hspace{.4cm}
       = -1 + \frac{1}{2}(\tilde{k}_\ell  + 1)
       = -1 - \frac{1}{2}|\tilde{k}_\ell  + 1|
       \le -1 - \epsilon_\ell  < -1. 
    \end{array}
  \end{equation}
  Similarly, 
   using $\epsilon_\ell \le \frac{1}{2}|\tilde{k}_\ell  - c_{2m, 2}|$ 
   and $\tilde{k}_\ell > c_{2m, 2}$, we can prove that $k_\ell > c_{2m, 2}$. 
  Therefore, $k_\ell \in (c_{2m, 2}, -1)$. 
  Moreover, since $c_{2m, 2} \in [-2, -1)$, it follows that for all
   $\ell \le n$, we have $|k_\ell| < 2$, and
  \begin{equation}
    \label{eq:xnbound}
    \renewcommand{\arraystretch}{1.2}
    \begin{array}{ll}
      |e_{n}| = |e_0| \prod_{\ell=0}^{n-1} |k_\ell| < 2^n |e_0| < 2^N |e_0| 
              = O(e_0).
    \end{array}
  \end{equation}

  Since $ n < N $, we have $ \tilde{k}_n \in \calB_{k, 2m} \subseteq (-2, -1)$.
  The function $h_{2m}(k)$ 
   is continuous at $\tilde{k}_n$,
   so the definition of continuity gives
  \[
      \forall \epsilon_{n+1} > 0, \exists \varepsilon_{n+1} > 0 
  \st \forall |k - \tilde{k}_n| < \varepsilon_{n+1},\ 
      |h_{2m}(k) - h_{2m}(\tilde{k}_{n})| < \frac{\epsilon_{n+1}}{2}.
  \]
  By the inductive hypothesis, for
  \[ 
      \varepsilon_{n+1} > 0, \exists \delta_{n+1, 1} > 0
  \st \forall |e_0| < \delta_{n+1, 1},\ 
      |k_n - \tilde{k}_n| < \varepsilon_{n+1},
  \]
   which implies
  \[
      \forall \epsilon_{n+1} > 0, \exists \delta_{n+1, 1} > 0
  \st \forall |e_0| < \delta_{n+1, 1},\ 
      |h_{2m}(k_n) - h_{2m}(\tilde{k}_{n})| < \frac{\epsilon_{n+1}}{2}.
  \]

  From \eqref{eq:deltaN+10} and \eqref{eq:knInTildeknDomain1}, we obtain 
   that $|e_0| < \delta_{n+1, 0} \leq \delta_n$ ensures
   $k_n < a_n := -1 -\epsilon_n < -1$. 
  Since $k_n > c_{2m, 2}$, we have 
   $a_n \in (k_n, -1) \subseteq (c_{2m, 2}, -1)$.
  For any $k \in (c_{2m, 2}, -1)$,  
   $1 - k^{2m-1} > 0$, $1 - k^{2m} < 0$, and
   $2mk^{2m-1} - (2m-1) k^{2m-2} - k^{4m-2} < 0$, 
   which, together with \eqref{eq:hform} and \eqref{eq:hprimeform}, yields
  \begin{displaymath}
  \renewcommand{\arraystretch}{1.2}
    \begin{array}{ll}
      h_{2m}(k) = \frac{1 - k^{2m-1}}{1 - k^{2m}} < 0, \quad
      h_{2m}'(k) = \frac{2mk^{2m-1} - (2m-1) k^{2m-2} 
                   - k^{4m-2}}{(1 - k^{2m})^2} 
                 < 0.
    \end{array}
  \end{displaymath}
  Hence, $h_{2m}(k)$ is strictly monotonically decreasing on $(c_{2m, 2}, -1)$.
  From $c_{2m, 2} < k_n < a_n < -1$, we have $h_{2m}(a_n) < h_{2m}(k_n) < 0$, 
   which implies
   $ \left|\frac{1 - {k}_{n}^{2m-1}}{1 - {k}_{n}^{2m}}\right| = |h_{2m}(k_n)| 
   < |h_{2m}(a_n)| = \left|\frac{1 - {a}_{n}^{2m-1}}{1 - {a}_{n}^{2m}}\right|$. 
  Combined with $\left|1 - {k}_{n}^{2m}\right| > a_n^{2m} - 1 > 0$, 
   Lemma~\ref{lem:epsLimit} and \eqref{eq:xnbound} give
  \begin{displaymath}
    \renewcommand{\arraystretch}{1.2}
    \begin{array}{ll}
      & \forall \epsilon_{n+1} > 0,\ \exists \delta_{n+1, 2} > 0 
    \st \forall |e_0| < \delta_{n+1, 2}, \\
      & \left|\frac{1 - {k}_{n}^{2m-1} + O(e_n)}{1 - {k}_{n}^{2m} + O(e_n)}  
      - \frac{1 - {k}_{n}^{2m-1}}{1 - {k}_{n}^{2m}}  \right|  
      = \left|\frac{1 - {k}_{n}^{2m-1} + O(e_0)}{1 - {k}_{n}^{2m} + O(e_0)}  
      - \frac{1 - {k}_{n}^{2m-1}}{1 - {k}_{n}^{2m}}  \right|  
      < \frac{\epsilon_{n+1}}{2}.
    \end{array}
  \end{displaymath}

  Therefore, for any $\epsilon_{n+1} > 0$, there exists 
   $\delta_{n+1} := \min(\delta_{n+1, 0}, \delta_{n+1, 1}, \delta_{n+1, 2}) 
                 > 0$ 
  such that when $|e_0| < \delta_{n+1}$, by the triangle inequality,
  \begin{displaymath}
    \renewcommand{\arraystretch}{1.2}
    \begin{array}{ll}
         & |k_{n+1} - \tilde{k}_{n+1}|
       = \left| k_{n+1}-h_{2m}(k_n)+h_{2m}(k_n)-\tilde{k}_{n+1} \right| 
      \\ & \hspace{2.2cm}
       = \left| \frac{1 - k_{n}^{2m-1} + O(e_n)}{1 - k_{n}^{2m} + O(e_n)} 
       - \frac{1 - {k}_{n}^{2m-1}}{1 - {k}_{n}^{2m}} 
       + h_{2m}(k_n) - h_{2m}(\tilde{k}_n)  \right| 
      \\ & \hspace{2.2cm}
       \le \left| \frac{1 - k_{n}^{2m-1} + O(e_n)}{1 - k_{n}^{2m} + O(e_n)} 
       - \frac{1 - {k}_{n}^{2m-1}}{1 - {k}_{n}^{2m}}  \right|
       + \left| h_{2m}(k_n)-h_{2m}(\tilde{k}_n)  \right|
      \\ & \hspace{2.2cm}
       < \frac{\epsilon_{n+1}}{2}+\frac{\epsilon_{n+1}}{2} 
       = \epsilon_{n+1},
    \end{array}
  \end{displaymath}
  which completes the induction step.
  
  Therefore, for $n = N$,
   since $\tilde{k}_N \notin \calB_{k, 2m}$ 
   and $\tilde{k}_N \notin \{-1, 0, 1, c_{2m, 1}, c_{2m, 2}\}$,
   we have $\tilde{k}_N \in (-\infty, c_{2m, 2}) \cup (-1, 0) \cup
                            (0, 1) \cup (1, +\infty)$. 
  On the one hand,                          
  taking $\epsilon_N = \min(|\tilde{k}_N - c_{2m, 2}|, |\tilde{k}_N + 1|, 
                            |\tilde{k}_N|, |\tilde{k}_N - 1|) > 0$, 
   there exists $\delta_N > 0$ such that when $|e_0| < \delta_N$, we have 
    $|k_N - \tilde{k}_N| < \epsilon_N $. 
  Then, following a derivation similar to \eqref{eq:knInTildeknDomain1},
    we obtain 
  \[k_N \in (-\infty, c_{2m, 2}) \cup (-1, 0) \cup (0, 1) \cup (1, +\infty).\]
  On the other hand, as in \eqref{eq:xnbound}, it holds 
   that $e_N < 2^N |e_0| = O(e_0)$ is sufficiently small.
  Consider taking $k_0 \leftarrow k_N, e_0 \leftarrow e_N$ 
   as new initial values for the iteration of the secant method. 
  Then $k_0, e_0$ satisfy the conditions 
   of Lemma~\ref{lem:knMoreRangeEven}, 
   and thus the error sequence $\{e_n\}$ converges linearly.
\end{proof}


%% file: sec/conclusion.tex

This paper has presented a novel approach
 for rigorously establishing the Q-order of convergence of
 the secant method for both simple and multiple roots.
In contrast to the proofs found in existing textbooks and literature,
 the method proposed herein offers greater rigor.
Notably, for the case of multiple roots,
 it provides sufficient conditions for the convergence of
 the secant method, thereby furnishing theoretical guidance
 for the selection of initial values in the secant method.
